\documentclass[a4paper, 10pt, american]{amsart}

\usepackage{times,latexsym,amssymb}
\usepackage{amsmath,amsthm,bm}
\usepackage{color}
\usepackage[colorlinks,pdfpagelabels,pdfstartview = FitH,bookmarksopen
= true,bookmarksnumbered = true,linkcolor = blue,plainpages =
false,hypertexnames = false,citecolor = red,pagebackref=false]{hyperref}

\usepackage{amsbsy}
\usepackage{amstext}
\usepackage{amssymb}
\usepackage{esint}
\usepackage{stmaryrd}
\usepackage[pdftex]{graphicx}
\usepackage{floatflt}
\usepackage{appendix}

\allowdisplaybreaks
\newtheorem{theorem}{Theorem}
\newtheorem{lemma}[theorem]{Lemma}
\newtheorem{definition}[theorem]{Definition}

\numberwithin{theorem}{section}
\numberwithin{equation}{section}

\def\R{\mathbb{R}}

\newcommand\sfb{{\boldsymbol{\mathsf b}}}

\newcommand\sfI{{\boldsymbol{\mathsf I}}}

\newcommand\sfm{{\boldsymbol{\mathsf m}}}
\newcommand\sfA{{\boldsymbol{\mathsf A}}}
\newcommand\sfE{{\boldsymbol{\mathsf E}}}
\newcommand\sfa{{\boldsymbol{\mathsf a}}}
\newcommand\sff{{\boldsymbol{\mathsf f}}}
\newcommand\sfK{{\boldsymbol{\mathsf K}}}
\newcommand\sfv{{\boldsymbol{\mathsf v}}}
\newcommand\sfF{{\boldsymbol{\mathsf F}}}

\renewcommand{\d}{\mathrm{d}}
\newcommand{\dx}{\mathrm{d}x}

\newcommand{\dt}{\mathrm{d}t}

\renewcommand{\epsilon}{\varepsilon}

\DeclareMathOperator{\Div}{div}

\renewcommand{\rho}{\varrho}

\def\eqn#1$$#2$${\begin{equation}\label#1#2\end{equation}}

\let\TeXchi\chi
\newbox\chibox
\setbox0 \hbox{\mathsurround0pt $\TeXchi$}
\setbox\chibox \hbox{\raise\dp0 \box 0 }
\def\chi{\copy\chibox}

\makeatletter
\@namedef{subjclassname@2020}{%
  \textup{2020} Mathematics Subject Classification}
\makeatother

\def\Xint#1{\mathchoice
    {\XXint\displaystyle\textstyle{#1}}%
    {\XXint\textstyle\scriptstyle{#1}}%
    {\XXint\scriptstyle\scriptscriptstyle{#1}}%
    {\XXint\scriptscriptstyle\scriptscriptstyle{#1}}%
    \!\int}
\def\XXint#1#2#3{\setbox0=\hbox{$#1{#2#3}{\int}$}
    \vcenter{\hbox{$#2#3$}}\kern-0.5\wd0}

\def\dashint{\Xint{\raise4pt\hbox to7pt{\hrulefill}}}

\def\Xiint#1{\mathchoice
    {\XXiint\displaystyle\textstyle{#1}}%
    {\XXiint\textstyle\scriptstyle{#1}}%
    {\XXiint\scriptstyle\scriptscriptstyle{#1}}%
    {\XXiint\scriptscriptstyle\scriptscriptstyle{#1}}%
    \!\iint}
\def\XXiint#1#2#3{\setbox0=\hbox{$#1{#2#3}{\iint}$}
    \vcenter{\hbox{$#2#3$}}\kern-0.5\wd0}
\def\biint{\Xiint{-\!-}}

\subjclass[2020]{35K40, 35K55, 35K65, 35K67}
\keywords{parabolic equations, gradient estimates}

\begin{document}
\title[Calder\'on-Zygmund estimates for parabolic $p$-Laplacian systems]{Calder\'on-Zygmund estimates for parabolic $p$-Laplacian systems with non-divergence form right-hand sides}
\date{\today}

\author[P. Andrade]{P\^edra Andrade}
\address{P\^edra Andrade\\
Fachbereich Mathematik, Universit\"at Salzburg\\
Hellbrunner Str. 34, 5020 Salzburg, Austria}
\email{pedra.andrade@plus.ac.at}

\author[V. B\"ogelein]{Verena B\"{o}gelein}
\address{Verena B\"ogelein\\
Fachbereich Mathematik, Universit\"at Salzburg\\
Hellbrunner Str. 34, 5020 Salzburg, Austria}
\email{verena.boegelein@plus.ac.at}

\author[F. Duzaar]{Frank Duzaar}
\address{Frank Duzaar\\
Fachbereich Mathematik, Universit\"at Salzburg\\
Hellbrunner Str. 34, 5020 Salzburg, Austria}
\email{frankjohannes.duzaar@plus.ac.at}

\author[K. Moring]{Kristian Moring}
\address{Kristian Moring\\
Fachbereich Mathematik, Universit\"at Salzburg\\
Hellbrunner Str. 34, 5020 Salzburg, Austria}
\email{kristian.moring@plus.ac.at}

\setcounter{tocdepth}{1}

\maketitle
\begin{center}
\textit{Dedicated to Antonia Passarelli di Napoli on her 60th birthday.}
\end{center}


\begin{abstract}
We establish local Calderón-Zygmund type estimates for weak solutions to nonlinear parabolic systems with $p$-growth and VMO coefficients. In particular, we prove that if the right-hand side belongs locally to $L^{\mu s}$, where the exponent $\mu$ depends explicitly on $p$, $N$, and a prescribed target exponent $s>p$, then the spatial gradient of the solution enjoys improved integrability $Du \in L^s_{\rm{loc}}$. 
The result provides a sharp transfer of integrability from the data to the gradient, consistent with the natural parabolic scaling, and recovers the optimal exponents in the linear case $p=2$. The proof combines intrinsic scaling techniques with a Calderón-Zygmund type iteration scheme.
\end{abstract}

\tableofcontents

\section{Introduction}
\subsection{Main results}
The aim of this paper is to establish sharp Calderón-Zygmund type gradient estimates for weak solutions to parabolic $p$-Laplace systems with ${\rm VMO}$-coefficients, namely systems of the form
\begin{equation}\label{par-sys}
    \partial_t u-\Div \big( \sfa (x,t)|Du|^{p-2}Du\big)=\sff, \qquad \mbox{with $p>\frac{2N}{N+2}$,}
\end{equation}
in a space-time cylinder $E_T:=E\times [0,T)$, where $E\subset\mathbb{R}^N$ is bounded and open set, $u\colon E_T\to \mathbb{R}^k$, and $\sff$ belongs to a suitable Lebesgue space. 
Throughout the paper, the coefficient function $\sfa\colon E_T\to\R$ is assumed to satisfy the uniform ellipticity bounds
\begin{equation} \label{bounds-a}
C_o \le \sfa(x,t) \le C_1
\qquad \text{for a.e.\ } (x,t)\in E_T.
\end{equation}
We further assume that $\sfa \in \mathrm{VMO}(E_T)$, in the sense that
\begin{equation}\label{a-vmo}
\lim_{R\downarrow 0} \sfv(R)=0,
\end{equation}
where
\[
\sfv(R)
:=
\sup_{\substack{Q_{\rho,\theta}(z_o)\Subset E_T\\ \rho\le R,\ \theta\le R^2}}
\biint_{Q_{\rho,\theta}(z_o)}
\big|
\sfa(x,t) - (\sfa)_{Q_{\rho,\theta}(z_o)}
\big|
\,\dx\dt,
\]
and
\[
(\sfa)_{Q_{\rho,\theta}(z_o)}
:=
\biint_{Q_{\rho,\theta}(z_o)} \sfa(x,t)\,\dx\dt.
\]
Our main result shows that an integrability assumption on the inhomogeneity $\sff$ transfers to the gradient of weak solutions to \eqref{par-sys}. More precisely, we prove the following result; see Definition~\ref{def:weak_solution} for the notion of weak solution.

\begin{theorem}\label{thm:main}
Let $N,k\in\mathbb{N}$ and let $p>\frac{2N}{N+2}$. 
Assume that $\sfa\colon E_T\to\mathbb{R}$ satisfies the $\mathrm{VMO}$-condition \eqref{a-vmo}. 
Let $s>p$ and define
\begin{equation}\label{def:mu}
\mu := \frac{N+2}{N(p-1)+p+s}.
\end{equation}
Suppose that $\sff\in L^{\mu s}_{\mathrm{loc}}(E_T,\mathbb{R}^k)$.
If
\[
u\in C^0\big((0,T);L^2_{\mathrm{loc}}(E,\mathbb{R}^k)\big)\cap 
L^p_{\mathrm{loc}}\big(0,T;W^{1,p}_{\mathrm{loc}}(E,\mathbb{R}^k)\big)
\]
is a local weak solution to the parabolic system~\eqref{par-sys} in the sense of Definition~\ref{def:weak_solution}, then 
\[
|Du|\in L^s_{\mathrm{loc}}(E_T). 
\]
Moreover, there exists a constant 
$
C = C(N,k,p,C_o,C_1,s,\sfv(\cdot))
$
such that for every parabolic cylinder $Q_R\Subset E_T$ we have
\begin{align*}
    \bigg[\biint_{Q_{\frac12 R}} &|D u|^s  \, \dx\dt\bigg]^{\frac1s} \\
    &\leq 
    C\Bigg[
    \biint_{Q_R}|Du|^p\,\dx\dt
    +
    \bigg[
    \biint_{Q_R}|R\sff|^{\mu s}\,\dx\dt + 1
    \bigg]^{\frac{1}{\mu s}\frac{p(N+2)}{p(N+2)-(N+p)}} \Bigg]^\frac{d}{p}, 
\end{align*}
where the
parabolic scaling deficit is defined by
\begin{align}\label{deficit}
    1
    \le 
    d
    :=
    \max \big\{ \tfrac12 p , \tfrac{2p}{p(N+2) - 2N}\big\}
    =
    \left\{
        \begin{array}{cl}
            \frac{p}{2}, & \mbox{if $p\ge 2$,} \\[5pt]
            \frac{2p}{p(N+2)-2N}, & \mbox{if $p<2$.}
        \end{array} 
    \right.
\end{align}
\end{theorem}
The exponents $\mu$ and $s$ are linked by \eqref{def:mu},
so that the condition $|\sff| \in L^{\mu s}_{\rm loc}(E_T)$ describes the integrability required by the datum in order to obtain $|Du| \in L^s_{\rm loc}(E_T)$. 
In this sense, the formulation prescribes a target exponent $s>p$, while the corresponding integrability of $\sff$ is determined implicitly.
Moreover, the condition $s>p$ is equivalent to
\[
\mu s > \frac{p(N+2)}{N(p-1)+2p},
\]
and thus characterizes the regime of genuine regularity improvement. The quantity $\mu s$ remains bounded by $N+2$, and the limiting case $\mu s \uparrow N+2$ corresponds to $s \uparrow \infty$.
In particular, the map $s \mapsto \mu s$ is increasing, with $\mu s \downarrow \frac{p(N+2)}{N(p-1)+2p}$ as $s \downarrow p$ and $\mu s \uparrow N+2$ as $s \uparrow \infty$.

In contrast to Theorem \ref{thm:main}, one may fix $q:=\mu s$ and express the gain of integrability in terms of $q$. This explicit formulation leads to the following theorem, which is equivalent to Theorem~\ref{thm:main}.

\begin{theorem}\label{thm:main-2}
Let $N,k\in\mathbb{N}$ and let $p>\frac{2N}{N+2}$. 
Assume that $\sfa\colon E_T\to\mathbb{R}$ satisfies the $\mathrm{VMO}$-condition \eqref{a-vmo}. Let
\begin{equation}\label{rel:s-q}
\frac{p(N+2)}{N(p-1)+2p} < q < N+2,
\qquad
s := \frac{q\,(N(p-1)+p)}{N+2 - q}.
\end{equation}
Assume that $\sff \in L^{q}_{\mathrm{loc}}(E_T,\mathbb{R}^k)$. If
\[
u\in C^0\big((0,T);L^2_{\mathrm{loc}}(E,\mathbb{R}^k)\big)\cap 
L^p_{\mathrm{loc}}\big(0,T;W^{1,p}_{\mathrm{loc}}(E,\mathbb{R}^k)\big)
\]
is a local weak solution in the sense of Definition~\ref{def:weak_solution}, then
\[
|Du|\in L^s_{\mathrm{loc}}(E_T).
\]
Moreover, there exists a constant 
$
C = C(N,k,p,C_o,C_1,q,\sfv(\cdot))
$
such that for every parabolic cylinder $Q_R\Subset E_T$ we have
\begin{align*}
\bigg[\biint_{Q_{\frac{1}{2}R}} &|Du|^s\,\dx\dt\bigg]^{\frac{1}{s}}\\
&\leq C \Bigg[
\biint_{Q_R} |Du|^p\,\dx\dt 
+ 
\bigg[
\biint_{Q_R} |R\sff|^{q}\,\dx\dt + 1
\bigg]^{\frac{1}{q}\cdot \frac{p(N+2)}{p(N+2)-(N+p)}}
\Bigg]^{\frac{d}{p}}.
\end{align*}
\end{theorem}
The integrability exponents in Theorem \ref{thm:main-2} are linked by the relation \eqref{rel:s-q},
which describes the transfer of integrability from the datum $|\sff| \in L^q_{\rm loc}(E_T)$ to the gradient $|Du| \in L^s_{\rm loc}(E_T)$. 
The condition
\[
q > \frac{p(N+2)}{N(p-1)+2p}
\]
is equivalent to $s>p$, and thus characterizes the regime of genuine regularity improvement. Note that $\frac{p(N+2)}{N(p-1)+2p}>1$ is automatically satisfied. 
The restriction $q<N+2$ guarantees that $s<\infty$.
In particular, the map $q \mapsto s$ is increasing, with $s \downarrow p$ as $q \downarrow \frac{p(N+2)}{N(p-1)+2p}$ and $s \uparrow \infty$ as $q \uparrow N+2$.

In the case $q > N+2$, the gradient becomes essentially bounded, that is, $|Du| \in L^\infty_{\rm loc}(E_T)$, provided the coefficients $\sfa$ are H\"older continuous, or more generally Dini-continuous; cf.~\cite[Chapter VIII]{DiBe} and \cite{BDGLS, Ku-Mi-MathAnn}. 
In fact, the borderline assumption that $\sff$ locally belongs to the Lorentz space $L^{N+2,1}$ already ensures that $Du \in L^\infty_{\rm loc}$; see \cite{Ku-Mi-MathAnn}.

\subsection{Sharpness of results}
In the linear case $p=2$, the exponent $s$ is sharp. Indeed, solutions to the heat equation admit the representation $Du \sim D\Gamma * \sff$, where $\Gamma$ denotes the heat kernel. Consequently, $|Du|$ can be estimated by a parabolic Riesz potential of order $1$, which satisfies the mapping property
\[
\sfI_1\colon L^q \to L^s, 
\qquad \frac{1}{s} = \frac{1}{q} - \frac{1}{N+2}.
\]
This yields
\[
s = \frac{q(N+2)}{N+2-q},
\]
which coincides with the exponent obtained above for $p=2$. The latter relation is optimal in the scale of Lebesgue spaces, and therefore the integrability gain cannot be improved in general.

The same reasoning applies in the case $p=2$ to nonlinear parabolic equations with Dini-continuous coefficients. For this class of equations, pointwise gradient estimates in terms of parabolic Riesz potentials have been established in~\cite{Du-Mi-pot}.

\subsection{Novelty and significance} In the stationary elliptic case of $p$-Laplace type equations and systems with divergence form right-hand side
\[
\Div \big( |Du|^{p-2}Du\big)
=
\Div \big( |\sfF|^{p-2}\sfF\big),
\qquad p>1,
\]
Calderón-Zygmund estimates of the form
\[
\sfF \in L^q_{\mathrm{loc}} (E,\R^{Nk}) \quad\Longrightarrow\quad Du \in L^q_{\mathrm{loc}}(E,\R^{Nk}),
\qquad q\ge p,
\]
are classical. The scalar case, i.e.~$k=1$, was treated in \cite{Iwaniec}, while the vectorial case $k>1$ goes back to \cite{DiBenedetto-Manfredi}. Non-divergence type inhomogeneities can be treated as a byproduct of Bogovski\u{\i}'s regularity result \cite{Bogo, Galdi} for equations of the form
\[
\Div \big( |\sfF|^{p-2}\sfF\big)=\sff,
\]
which yields the implication
\[
\sff \in L^{\frac{Ns}{N(p-1)+s}}_{\mathrm{loc}}(E,\R^k)
\quad\Longrightarrow\quad
Du \in L^s_{\mathrm{loc}}(E,\R^{Nk}),
\qquad s \ge p.
\]
In particular, there is no need to develop a separate theory for non-divergence type inhomogeneities, as the sharp exponents follow directly from Bogovski\u{\i}'s construction.

Extensions of Calder\'on-Zygmund estimates to anisotropic equations with $\mathrm{VMO}$-coefficients can be found in \cite{Kinnunen-Zhou-1, Kinnunen-Zhou-2}, while equations and systems with nonstandard $p(x)$-growth conditions are treated in \cite{Acerbi-Min-p(x)}; see also \cite{Mingione:Measure-data} for a complete Calder\'on-Zygmund theory for measure data problems. Finally, nonlocal $p$-Laplace equations with H\"older continuous coefficients $\sfa$ and non-divergence type right-hand side $\sff$ have been studied in \cite{BDLM-1, BDLM-2}. More precisely, for the $(\sigma,p)$-Poisson equation with Hölder-continuous coefficients 
\[
(-\sfa\Delta_p)^\sigma u = \sff \quad \text{in } E,
\]
we have the implication
\[
\sff \in L^{\frac{Ns}{N(p-1)+s(1-p(1-\sigma))}}_{\mathrm{loc}}(E)
\quad\Longrightarrow\quad
D u \in L^s_{\mathrm{loc}}(E,\R^N),
\qquad s \ge p.
\]
It is worth noting that, formally, the integrability exponent of the inhomogeneity coincides with the one from the local case 
when $\sigma=1$. In contrast to the local case, divergence-type right-hand sides are not natural in the nonlocal setting.

The parabolic case of $p$-Laplace type equations and systems remained open for a long time and was settled by Acerbi and Mingione in~\cite{Acerbi-Min}. They showed that weak solutions to
\[
\partial_t u - \Div \big( \sfa(x,t)|Du|^{p-2}Du \big)
=
\Div \big( |\sfF|^{p-2}\sfF \big),
\qquad p>\frac{2N}{N+2},
\]
with $\mathrm{VMO}$-coefficients $\sfa$ satisfy Calder\'on-Zygmund type estimates of the form
\[
\sfF \in L^q_{\mathrm{loc}}(E_T,\R^{Nk})
\quad\Longrightarrow\quad
Du \in L^q_{\mathrm{loc}}(E_T,\R^{Nk}),
\qquad q \ge p.
\]
Their approach, inspired by \cite{Caff-Peral, Peral-Soria}, relies on intrinsic scaling and Calder\'on-Zygmund type covering arguments, thereby circumventing the use of harmonic analysis tools such as singular integrals and maximal functions, and paved the way for further developments in the theory of nonlinear $p$-parabolic systems. In particular, more general systems of the form
\[
\partial_t u - \Div \sfA(x,t,Du)=\Div \big(|\sfF|^{p-2}\sfF\big)
\]
were treated in \cite{Du-Mi-St}. Here, the vector field $\sfA$ is assumed to satisfy standard $p$-growth and ellipticity conditions, as well as a continuity condition in $x$ normalized by the natural $p$-growth. In the case $p\ge 2$, a nonlinear Calder\'on-Zygmund theory is developed, yielding the integrability transfer
\[
\sfF \in L^q_{\mathrm{loc}}(E_T,\R^{Nk})
\quad\Longrightarrow\quad
Du \in L^q_{\mathrm{loc}}(E_T,\R^{Nk}),
\qquad p \le q < p^\sharp,
\]
for some $p^\sharp >p+\frac4N$. The singular subquadratic regime $\frac{2N}{N+2}<p<2$ was treated in \cite{Scheven-1}.
Subsequently, the theory has been extended to various other settings, including parabolic systems with nonstandard $p(x,t)$-growth~\cite{Ba-Bo}, obstacle problems~\cite{BDM-obst, By-Cho}, and global estimates~\cite{Bo-global, By-Oh-Wa, Cho-Song-Ryu}. For related results, such as nonlinear potential estimates, see also \cite{Ba-Ha, DiKiLeNo, Ku-Mi-Ibero, Ku-Mi-ARMA, Ku-Mi-JEMS, Mingione:BUMI, Mingione:Sketches} and the references therein. We do not attempt to provide a complete account of the vast literature on the subject and refer the reader to the above references and the works cited therein. Although the above results cover a wide range of settings, most of the previously established parabolic estimates are restricted to right-hand sides in divergence form. To the best of our knowledge, optimal Calder\'on-Zygmund bounds for non-divergence form right-hand sides are known only in the case $p=2$, see \cite{Du-Mi-pot, By-Ki-Ku}.
For $\frac{2~}{N+2}<p\ne 2$, one has the implication
$$
    \sff \in L^{\frac{(N+2)s}{N(p-1)+2p}}_{\mathrm{loc}}(E_T,\R^k)
    \quad\Longrightarrow\quad
    Du \in L^s_{\mathrm{loc}}(E_T,\R^{Nk}),
    \qquad s \ge p,
$$
established in \cite{Bo-global, By-Wo}. This estimate is weaker than that obtained in Theorem~\ref{thm:main} and, in particular, does not recover the optimal result in the case $p=2$. 

In contrast to the elliptic setting, the Bogovski\u{\i} reduction from non-divergence form right-hand sides to divergence form is not available in the parabolic case. The aim of the present paper is to develop an alternative approach leading to optimal Calder\'on-Zygmund estimates for parabolic $p$-Laplacian equations and systems with $\mathrm{VMO}$-coefficients and inhomogeneities in non-divergence form.

\subsection{Strategy of proof}
The overall strategy of the proof follows the approach introduced in~\cite{Acerbi-Min}, and in particular avoids the use of maximal function techniques. A key observation is that solutions to the associated homogeneous problem with constant coefficients have a bounded spatial gradient. 

The proof is based on a two-step comparison argument, which allows to transfer the maximal amount of integrability to the original solution. More precisely, we first compare the given solution with the solution to the associated homogeneous problem and establish suitable estimates for their difference. At this stage, we employ a Gagliardo-Nirenberg type inequality in order to keep the integrability exponent of the source term $\sff$ as low as possible. We then compare the solution of the homogeneous system with that of the corresponding homogeneous system with frozen coefficients. 
Combined with a stopping-time argument, this leads to estimates for the superlevel sets of the spatial gradient $|Du|$. Owing to the non-homogeneous nature of the parabolic equation in space and time, these estimates are carried out on suitably scaled cylinders, the so-called intrinsic cylinders; see~\cite{DiBe, DiBenedetto_Holder}. More precisely, on a family of intrinsic space-time cylinders $Q_i$, all contained in a larger cylinder $Q$ and constructed via a Vitali-type covering argument, we obtain an estimate of the form
\begin{align*}
    \iint_{Q_i \cap \{|Du| \gtrsim \lambda\}}|Du|^p\,\dx\dt
    \lesssim
    \left[
    \iint_{Q_i} |\sff|^{\frac{p(N+2)}{N(p-1)+2p}}\,\dx\dt
    \right]^{\frac{N(p-1)+2p}{N(p-1)+p}}
    + \dots .
\end{align*}
First, observe that the exponent of the $\sff$-integral is strictly larger than one. To obtain the desired $L^s$-integrability of $|Du|$, this exponent must be increased by a factor $\frac{s}{p}$. On the other hand, applying the same multiplicative factor to the exponent of $|\sff|$ would yield
\[
\frac{p(N+2)}{N(p-1)+2p}\cdot \frac{s}{p}
=
\frac{s(N+2)}{N(p-1)+2p}
>
\frac{s(N+2)}{N(p-1)+p+s}
=
\mu s,
\]
which exceeds the exponent obtained in Theorem~\ref{thm:main}. Recall that $\mu$ is defined in~\eqref{def:mu}. The above difficulty is resolved by an interpolation argument. More precisely, an application of H\"older's inequality yields
\begin{align*}
    \left[
    \iint_{Q_i} |\sff|^{\frac{p(N+2)}{N(p-1)+2p}}\,\dx\dt
    \right]^{\frac{N(p-1)+2p}{N(p-1)+p}} 
    \le
    \bigg[
    \underbrace{\iint_{Q_i} |\sff|^{\mu s}\,\dx\dt}_{\le\, \iint_{Q} |\sff|^{\mu s}\,\dx\dt}
    \bigg]^{\frac{p}{N(p-1)+p}}
    \iint_{Q_i} |\sff|^{\mu p}\,\dx\dt.
\end{align*}
The resulting integral inequality,
\begin{align*}
    \iint_{Q_i \cap \{|Du| \gtrsim \lambda\}}|Du|^p\,\dx\dt
    \lesssim
    \|\sff\|_{L^{\mu s}(Q)}^{\frac{\mu sp}{N(p-1)+p}}
    \iint_{Q_i} |\sff|^{\mu p}\,\dx\dt
    + \dots,
\end{align*}
is thus perfectly balanced and stable under summation over disjoint families of cylinders. In particular, enlarging both the exponent in the $|Du|$-integral and that in the second $|\sff|$-integral by a factor $\frac{s}{p}$ heuristically yields the desired integrability exponents. This heuristic can be made rigorous by means of a covering argument, which allows one to pass to an integral inequality for the superlevel sets of $|Du|$ and $|\sff|$ in the larger cylinder $Q$. The resulting inequality then serves as the starting point for a standard Fubini-type argument leading to the final gradient estimate.

\section{Notation, definitions and tools}\label{SS:2}
\subsection{Notation}
Throughout the paper, we denote by $\mathbb{N}_0:=\mathbb{N}\cup\{0\}$ the set of nonnegative integers, and by $\mathbb{R}^k$ and $\mathbb{R}^N$ the standard Euclidean spaces. Given a bounded open set $E\subset\mathbb{R}^N$ and a time horizon $T>0$, we consider the space-time cylinder $E_T:=E\times [0,T)$. For a vector-valued function $u\colon E_T\to\mathbb{R}^k$, we write $Du$ for the spatial gradient and $\partial_t u$ for the time derivative.

Points in space-time are denoted by $z_o=(x_o,t_o)\in\mathbb{R}^N\times\mathbb{R}$. 
For $\rho>0$, we write $B_\rho(x_o)\subset\mathbb{R}^N$ for the open ball centered at $x_o$ with radius $\rho$, and simply $B_\rho$ if $x_o=0$. 
Given $\rho,\theta>0$, we define the symmetric parabolic cylinder
\[
Q_{\rho,\theta}(z_o):=B_\rho(x_o)\times (t_o-\theta,t_o+\theta).
\]
Occasionally, we use the abbreviation
\[
\sigma Q_{\rho,\theta}(z_o):=Q_{\sigma\rho,\sigma\theta}(z_o),
\quad \sigma>0.
\]
For $\lambda>0$ and $\rho>0$, we define the intrinsic parabolic cylinder centered at $z_o=(x_o,t_o)$ by
\[
Q_\rho^{(\lambda)}(z_o)
:=
B_\rho(x_o)\times
\bigl(t_o-\lambda^{2-p}\rho^2,\;t_o+\lambda^{2-p}\rho^2\bigr).
\]
When $\lambda=1$, the superscript is omitted, and the cylinders coincide with the standard parabolic cylinders. If $z_o=0$ or if no confusion arises, we also omit the reference to $z_o$.

\subsection{Notion of weak solutions}
Having fixed the notation and the underlying geometry, we now turn to the notion of weak solutions, which will be used throughout the paper.

\begin{definition}\label{def:weak_solution}\upshape
Assume that $\sfa:E_T\to\R$ satisfies \eqref{bounds-a}, and that $\sff\in L^1_{\mathrm{loc}}(E_T,\R^k)$. 
A function
\[
u\in C^0\big((0,T); L^2_{\mathrm{loc}}(E,\R^k)\big)\cap
L^p_{\mathrm{loc}}\big(0,T;W^{1,p}_{\mathrm{loc}}(E,\R^k)\big)
\]
is called a local weak solution to \eqref{par-sys} if
\begin{align}\label{weak-solution}
	\iint_{E_T}\big[u\cdot\varphi_t - \sfa(x,t)|Du|^{p-2}Du\cdot D\varphi\big]\dx\dt
    =
    \iint_{E_T}\sff\cdot\varphi\,\dx\dt
\end{align}
holds for every test function $\varphi\in C_0^\infty(E_T,\R^k)$.
\hfill$\Box$
\end{definition}

\subsection{Tools}
We begin by recalling a higher integrability result for the spatial gradient under a subintrinsic condition. It can be deduced from~\cite{Kinnunen-Lewis:1} by an argument similar to that in~\cite[Lemma~3]{Acerbi-Min}. For completeness, we include the proof.

\begin{lemma}\label{lem:HI}
Let $p>\frac{2N}{N+2}$, $0<C_o\le C_1$, and $\sfK\ge1$. 
Then there exist constants $\sigma=\sigma(N,k,p,C_o,C_1)>p$ and 
$C=C(N,k,p,C_o,C_1,\sfK)\ge1$ with the following property. 
Suppose that 
\[
v\in C^0\big((0,T);L^2_{\mathrm{loc}}(E,\R^k)\big)
\cap L^p_{\mathrm{loc}}\big(0,T;W^{1,p}_{\mathrm{loc}}(E,\R^k)\big)
\]
is a local weak solution to \eqref{par-sys} with $\sff=0$ in $E_T$, i.e.~to the associated homogeneous system.
Assume moreover that for some cylinder 
$Q_\rho^{(\lambda)}(z_o) \Subset E_T$ the 
subintrinsic condition
\[
\bigg[
\biint_{Q_\rho^{(\lambda)}(z_o)}
|Dv|^p\,\dx\dt\bigg]^\frac1p
\le \sfK \lambda
\]
holds for some $\lambda > 0$.
Then
\[
\bigg[
\biint_{\frac12 Q^{(\lambda)}_\rho(z_o)}
|Dv|^{\sigma}\,\dx\dt\bigg]^\frac{1}{\sigma}
\le
C \lambda.
\]
\end{lemma}

\begin{proof}
We proceed as in \cite[Lemma~3]{Acerbi-Min}. Without loss of generality we assume $z_o=0$. We consider the re-scaled maps 
\begin{equation*}
    \tilde v(x,t)
    :=
    \frac{v(\rho x,\lambda^{2-p}\rho^2 t)}{\rho\lambda},
    \qquad
    \tilde \sfa(x,t)
    :=
    \sfa(\rho x,\lambda^{2-p}\rho^2 t),
\end{equation*}
with $(x,t)\in Q_1$. Then, we have 
$\tilde v\in C^0(-1,1;L^2(B_1,\R^k))\cap L^p(-1,1;W^{1,p}(B_1,\R^k))$ is a weak solution to 
\begin{equation*}
    \partial_t\tilde v -
    \Div \big( \tilde\sfa (x,t)|D\tilde v|^{p-2}D\tilde v\big)
    =
    0
    \qquad\mbox{in $Q_1$.}
\end{equation*}
Then, \cite{Kinnunen-Lewis:1} or \cite[Theorem~2.2]{BD} ensures the existence of $\sigma>p$ such that 
\begin{align*}
    \bigg[\biint_{Q_{\frac12}} |D\tilde v|^{\sigma}\,\dx\dt \bigg]^{\frac1{\sigma}}
    \le 
    C\Bigg[\bigg[\biint_{Q_1} |D\tilde v|^{p}\,\dx\dt \bigg]^{\frac{p+(\sigma-p)d}{\sigma p}} +1\Bigg],
\end{align*}
where $d$ is defined in~\eqref{deficit}. 
Scaling back to $v$ yields
\begin{align*}
    \bigg[\biint_{\frac12 Q^{(\lambda)}_{\rho}} |Dv|^{\sigma}\,\dx\dt \bigg]^{\frac1{\sigma}}
    &\le 
    C\lambda\Bigg[\lambda^{-\frac{p+(\sigma-p)d}{\sigma}}\bigg[\biint_{Q^{(\lambda)}_{\rho}} |Dv|^{p}\,\dx\dt \bigg]^{\frac{p+(\sigma-p)d}{\sigma p}} +1\Bigg] \\
    &\le 
    C\lambda\Big[\sfK^{\frac{p+(\sigma-p)d}{\sigma}} +1\Big].
\end{align*}
This proves the claim. 
\end{proof}

We next recall a local $L^\infty$-estimate for the spatial gradient of solutions to the homogeneous system under a subintrinsic condition. The result can be deduced from \cite[Chapter VIII, Theorems 5.1 and 5.2]{DiBe}; see also \cite[Lemmas 1 and 2]{Acerbi-Min} for a proof. The only modification concerns the choice of parameters in \cite[Lemma 2]{Acerbi-Min}, where one has to take $\gamma=\rho$ and $\theta=\lambda^{2-p}\rho^2$ in accordance with our definition of subintrinsic cylinders in the subquadratic case.

\begin{lemma}\label{lem:-apriori-p>2}
Let $p>\frac{2N}{N+2}$, $0 < C_o \le C_1$, and $\sfK > 0$. 
Then there exists a constant 
\[
\sfA  = \sfA(N,k,p,C_o,C_1,\sfK)
\]
with the following property. Suppose that
\[
w \in C^0\big((0,T);L^2_{\mathrm{loc}}(E,\R^k)\big)
\cap L^p_{\mathrm{loc}}\big(0,T;W^{1,p}_{\mathrm{loc}}(E,\R^k)\big)
\]
is a local weak solution to
\[
\partial_t w 
- \Div\big(\sfb |Dw|^{p-2}Dw\big) = 0
\qquad \text{in } E_T,
\]
where $\sfb \in [C_o,C_1]$ is a constant. 
Assume moreover that for some cylinder 
$Q_\rho^{(\lambda)}(z_o) \Subset E_T$ the 
subintrinsic condition
\[
\bigg[\biint_{Q_\rho^{(\lambda)}(z_o)}
|Dw|^p\,\dx\dt\bigg]^\frac1p
\le \sfK \lambda
\]
holds for some $\lambda > 0$. Then, we have
\[
\sup_{\frac{1}{2}Q_\rho^{(\lambda)}(z_o)}
|Dw|
\le
\sfA \lambda .
\]
\end{lemma}

The next lemma states a standard algebraic estimate. 

\begin{lemma}\label{lem:alg-1}
Let $p>1$. Then there exists a constant $C=C(N,k,p)>0$ such that for every $\xi,\zeta\in\mathbb{R}^{kN}$ one has
\begin{equation*}
    |\xi|^p\le C|\zeta|^p +
    C\big(
    |\xi|^2+|\zeta|^2
    \big)^\frac{p-2}2 |\zeta - \xi|^2.
\end{equation*}
\end{lemma}

\section{Comparison estimates}
The gain in integrability is obtained by a two-step comparison argument. 
We first compare the solution to the inhomogeneous system with the solution 
to the associated homogeneous problem (i.e., with vanishing right-hand side), 
which yields higher integrability of the gradient. 
In a second step, we compare this solution with that of the corresponding 
frozen-coefficient system, whose gradient is locally bounded. 
The first comparison step is quantified in the following lemma, while the second one will be performed in Section~\ref{sec:comp2}. 
Combining them allows us to transfer the improved regularity of the gradients 
of the auxiliary solutions into a quantitative higher integrability of the 
gradient of the original solution.
For the first comparison step, we suppose that 
\[
u \in C^0\big((0,T);L^2_{\mathrm{loc}}(E,\R^k)\big)
\cap L^p_{\mathrm{loc}}\big(0,T;W^{1,p}_{\mathrm{loc}}(E,\R^k)\big)
\]
is a local weak solution to \eqref{par-sys}, and let 
$Q_{\rho,\theta}(z_o) \Subset E_T$.
Denote by 
\begin{equation*}
    v\in C^0\big([t_o-\theta, t_o+\theta];
    L^2(B_\rho (x_o);\R^k)\big)
    \cap
    L^p\big(t_o-\theta, t_o+\theta; W^{1,p}
    (B_\rho (x_o);\R^k)\big)
\end{equation*}
the unique weak solution to the Cauchy–Dirichlet problem
\begin{equation}\label{CD-homo}
    \left\{
    \begin{array}{cl}
    \partial_t v-\Div \big( \sfa (x,t)|Dv|^{p-2}Dv\big)=0,
    &\mbox{in $Q_{\rho,\theta}(z_o)$,}\\[7pt]
    v=u,
    &\mbox{on $\partial_{\rm par} Q_{\rho,\theta}(z_o)$,}
    \end{array}
    \right.
\end{equation}
where $\sfa$ satisfies \eqref{bounds-a}.

\begin{lemma}[First comparison estimate]\label{lem:fract-level-comparison}
Let $p>\frac{2N}{N+2}$ and suppose that $u$ is a local weak solution to~\eqref{par-sys} 
in the sense of Definition~\ref{def:weak_solution}, and let $Q_{\rho,\theta} \equiv Q_{\rho,\theta}(z_o) \Subset E_T$. Denote by $v$ the unique weak solution to the homogeneous Cauchy-Dirichlet problem~\eqref{CD-homo} on $Q_{\rho,\theta}$ and set
$$
    q=p\frac{N+2}{N}, \qquad q'=\frac{p(N+2)}{p(N+2)-N}. 
$$
Then, there exists a constant $C = C(N,p,C_o) > 0$ such that the following estimates hold.

\medskip
\noindent
\emph{(i) Superquadratic case.} 
If $p\in[2,\infty)$, then
\begin{align*}
\sup_{\tau\in (t_o-\theta, t_o+\theta)}
\int_{B_\rho\times\{\tau\}}
|u-v|^2\,\dx
&+
\iint_{Q_{\rho,\theta}}
|Du-Dv|^p\,\dx\dt\\
&\le
    C
\bigg[
\iint_{Q_{\rho,\theta}}
|\sff|^{q'}
\,\dx\dt
\bigg]^{\frac1{q'}\frac{p(N+2)}{p(N+2)-N-p}} .
\end{align*}

\medskip
\noindent
\emph{(ii) Subquadratic case.} 
If $p\in(\frac{2N}{N+2},2]$,  then for every $\delta\in(0,1]$ we have
\begin{align*}
\sup_{\tau\in (t_o-\theta, t_o+\theta)}&
\int_{B_\rho\times\{\tau\}}
|u-v|^2\,\dx
+
\iint_{Q_{\rho,\theta}}
|Du-Dv|^p\,\dx\dt\\
&\le
\delta
\iint_{Q_{\rho,\theta}}
|Du|^p\,\dx\dt
+
\frac{C}{\delta^{q'\frac{2-p}{p}}}
\bigg[
\iint_{Q_{\rho,\theta}}
|\sff|^{q'}
\,\dx\dt
\bigg]^{\frac1{q'}\frac{p(N+2)}{p(N+2)-N-p}} .
\end{align*}
\end{lemma}

\begin{proof} Throughout the proof we suppress the dependence on the center $z_o$. We observe that $w:=u-v$ is a weak solution to
\begin{equation}
    \left\{
    \begin{array}{cl}
    \partial_t (u-v)-\Div \big( \sfa (x,t)\big(|Du|^{p-2}Du -|Dv|^{p-2}Dv\big)\big)=\sff,
    &\mbox{in $Q_{\rho,\theta}$,}\\[7pt]
    u-v=0,
    &\mbox{on $\partial_{\rm par} Q_{\rho,\theta}$.}
    \end{array}
    \right.
\end{equation}
Let $\tau \in (t_o - \theta , t_o + \theta)$ and $\varepsilon \in (0, t_o + \theta - \tau)$. Testing formally with
\[
\varphi := \zeta(t)(u-v),
\]
a procedure that can be justified rigorously by means of a time regularization via Steklov averages, 
where $\zeta \colon [t_o-\theta, t_o+\theta] \to [0,1]$ is the usual time cut-off function satisfying
\[
\zeta \equiv 1 \text{ on } [t_o-\theta,\tau], 
\qquad
\zeta \equiv 0 \text{ on } [\tau+\varepsilon, t_o+\theta],
\]
and affine on $[\tau,\tau+\varepsilon]$, we obtain
\begin{align*}
\iint_{Q_{\rho,\theta}}
\partial_t w \cdot\zeta w
\, \dx\dt
&+
\iint_{Q_{\rho,\theta}}\zeta
\sfa(x,t)
\big(
|Du|^{p-2}Du - |Dv|^{p-2}Dv
\big)
\cdot
Dw 
\, \dx\dt\\
&=
\iint_{Q_{\rho,\theta}}\zeta\sff\cdot w\,\dx\dt.
\end{align*}
Integrating by parts in time and using that $w=0$ on 
$\partial_{\mathrm{par}} Q_{\rho,\theta}$, we obtain
\begin{align*}
&\tfrac1{2\epsilon}
\iint_{B_\rho\times (\tau, \tau +\epsilon)}
|w|^2\,\dx\dt\\
&\quad\phantom{=\,}+
\iint_{B_\rho\times (t_o-\theta,\tau+\epsilon)}
\zeta \,
\sfa(x,t)
\big(
|Du|^{p-2}Du - |Dv|^{p-2}Dv
\big)
\cdot
(Du-Dv)
\, \dx\dt
\\
&\quad
=
\iint_{B_\rho\times (t_o-\theta,\tau+\epsilon)}
\zeta \,\sff\cdot w
\, \dx\dt .
\end{align*}
Taking into account that $\sfa \ge C_o$, the monotonicity of the map 
$\xi \mapsto |\xi|^{p-2}\xi$, and that $0 \le \zeta \le 1$, we obtain, 
letting $\varepsilon \downarrow 0$,
\begin{align*}
&\tfrac1{2}
\int_{B_\rho\times\{\tau\}}
|u-v|^2\,\dx\\
&\quad\phantom{=\,}+
C_o\iint_{B_\rho\times (t_o-\theta,\tau)}
\underbrace{\big(
|Du|^{p-2}Du - |Dv|^{p-2}Dv
\big)
\cdot
(Du-Dv)}_{\ge C(p)(|Du|^2+|Dv|^2)^\frac{p-2}2|Du-Dv|^2}
\, \dx\dt
\\
&\quad
\le
\iint_{B_\rho\times (t_o-\theta,\tau)}
|\sff|| w|
\, \dx\dt .
\end{align*}
Taking the supremum with respect to $\tau$ in the first term 
on the left-hand side and letting $\tau \uparrow t_o+\theta$ in the second, 
we obtain
\begin{align*}
 \sfI_1
&+
2C_oC(p)\sfI_2
\le 
2\sfI\sfI,
\end{align*}
where
\[
    \sfI_1
    :=
    \sup_{\tau\in (t_o-\theta, t_o+\theta)}
\int_{B_\rho\times\{\tau\}}
|u-v|^2\,\dx,
\]
\[
\sfI_2:=\iint_{Q_{\rho,\theta}}
\big(|Du|^2+|Dv|^2\big)^\frac{p-2}2|Du-Dv|^2
\, \dx\dt,
\]
and
\[
    \sfI\sfI:= \iint_{Q_{\rho,\theta}}
    |\sff|| u-v|\, \dx\dt.
\]
It remains to estimate the right-hand side $\sfI\sfI$. 
By Hölder's inequality and the Gagliardo- Nirenberg inequality \cite[Chapter I, Proposition 3.1]{DiBe}, we obtain
\begin{align*}
    \sfI\sfI
    &\le
    \bigg[\iint_{Q_{\rho,\theta}}|\sff|^{q'}
    \,\dx\dt\bigg]^\frac1{q'}
    \bigg[
    \iint_{Q_{\rho,\theta}}
    |u-v|^q\,\dx\dt
    \bigg]^\frac{1}{q}\\
    &\le C
    \bigg[\iint_{Q_{\rho,\theta}}|\sff|^{q'}
    \,\dx\dt\bigg]^\frac1{q'}
    \cdot
    \bigg[
    \iint_{Q_{\rho,\theta}}|Du-Dv|^p\,\dx\dt
    \bigg]^\frac{N}{p(N+2)}\sfI_1^\frac{1}{N+2}.
\end{align*}
We next apply Young's inequality for triple products in the form
\begin{equation*}
abc \le \varepsilon \big(a^{\alpha} + b^{\beta}\big)
+ \varepsilon^{1-\gamma} c^{\gamma},
\qquad
\frac1{\alpha}+\frac1{\beta}+\frac1{\gamma}=1,
\end{equation*}
with the choice
\[
\alpha = N+2, 
\qquad 
\beta = q,
\qquad 
\gamma = \frac{p(N+2)}{p(N+2)-N-p},
\]
and the natural identification of $a$, $b$, and $c$ with the corresponding
factors in the preceding estimate.
Applying this estimate to the right-hand side yields
\begin{align*}
\sfI\sfI
&\le
\varepsilon
\left[
\sfI_1
+
\iint_{Q_{\rho,\theta}}
|Du-Dv|^p\,\dx\dt
\right]
+
\frac{C}{\varepsilon^{\frac{N+p}{p(N+2)-N-p}}}
\left[
\iint_{Q_{\rho,\theta}}
|\sff|^{q'}
\,\dx\dt
\right]^{\frac{p(N+2)-N}{p(N+2)-N-p}} .
\end{align*}
We now distinguish between the super- and subquadratic cases. 
If $p \ge 2$, we may estimate $\sfI_2$ from below by
\[
\sfI_2
\ge
2^{\frac{2-p}{2}}
\iint_{Q_{\rho,\theta}}
|Du-Dv|^p\,\dx\dt .
\]
Combining this with the estimate of the right-hand side and setting
\[
C := \min\big\{1,\, C_o C(p)\big\},
\]
we immediately obtain
\begin{align*}
C
\left[
\sfI_1
+
\iint_{Q_{\rho,\theta}}
|Du-Dv|^p\,\dx\dt
\right]
&\le
2\varepsilon
\left[
\sfI_1
+
\iint_{Q_{\rho,\theta}}
|Du-Dv|^p\,\dx\dt
\right]
\\
&\phantom{\le\,}
+
\frac{C}{\varepsilon^{\frac{N+p}{p(N+2)-N-p}}}
\left[
\iint_{Q_{\rho,\theta}}
|\sff|^{q'}
\,\dx\dt
\right]^{\frac{p(N+2)-N}{p(N+2)-N-p}} .
\end{align*}
Choosing $\varepsilon = \frac{1}{2}C$, the first term on the right-hand side 
can be absorbed into the left-hand side. This yields the claimed comparison 
estimate, namely
\begin{align*}
\sfI_1
+
\iint_{Q_{\rho,\theta}}
|Du-Dv|^p\,\dx\dt
&\le
    C
\left[
\iint_{Q_{\rho,\theta}}
|\sff|^{q'}
\,\dx\dt
\right]^{\frac{p(N+2)-N}{p(N+2)-N-p}} .
\end{align*}
In the subquadratic case $p<2$ the argument is slightly different. 
We first rewrite $|Du-Dv|^p$ by splitting the factor $1$ in a suitable way. 
More precisely, for $\delta>0$ we write
\begin{align*}
|Du-Dv|^p
&=
\delta^{\frac{2-p}{2}}
\big(|Du|^2+|Dv|^2\big)^{\frac{(2-p)p}{4}}
\delta^{-\frac{2-p}{2}}
\big(|Du|^2+|Dv|^2\big)^{\frac{(p-2)p}{4}}
|Du-Dv|^p .
\end{align*}
We then apply Young's inequality with exponents 
$\frac{2}{2-p}$ and $\frac{2}{p}$ to the two resulting factors. 
This yields
\begin{align*}
|Du-Dv|^p
\le
\delta \big(|Du|^2+|Dv|^2\big)^{\frac{p}{2}}
+
\delta^{-\frac{2-p}{p}}
\big(|Du|^2+|Dv|^2\big)^{\frac{p-2}{2}}
|Du-Dv|^2 .
\end{align*}
To continue, we observe that
\begin{align*}
    \big(|Du|^2+|Dv|^2\big)^{\frac{p}{2}}
    &\le
    |Du|^p+|Dv|^p
    \le
    3|Du|^p+ 2|Du-Dv|^p.
\end{align*}
Inserting this estimate into the previous inequality and restricting 
$\delta \le \frac14$, we may absorb the term 
$\frac12 |Du-Dv|^p$ on the right-hand side into the left-hand side. 
We thus obtain
\begin{align*}
|Du-Dv|^p
\le
6 \delta |Du|^p
+
2\delta^{-\frac{2-p}{p}}
\big(|Du|^2+|Dv|^2\big)^{\frac{p-2}{2}}
|Du-Dv|^2.
\end{align*}
Integrating over $Q_{\rho,\theta}$ yields
\begin{align*}
\iint_{Q_{\rho,\theta}}|Du-Dv|^p\,\dx\dt
&\le
6 \delta \iint_{Q_{\rho,\theta}}|Du|^p\,\dx\dt
+
2\delta^{-\frac{2-p}p}\sfI_2.
\end{align*}
It follows that
\begin{align*}
    \sfI_1 &+\iint_{Q_{\rho,\theta}}|Du-Dv|^p\,\dx\dt\\
    &\le
    6 \delta \iint_{Q_{\rho,\theta}}|Du|^p\,\dx\dt
    +
    \sfI_1
    +
     2\delta^{-\frac{2-p}{p}}
    \sfI_2\\
    &\le
     6 \delta \iint_{Q_{\rho,\theta}}|Du|^p\,\dx\dt
    +
    2\sfI\sfI
    +
    \frac{\sfI\sfI}{C_oC(p)\delta^{\frac{2-p}{p}}} \\
    &\le
     6 \delta \iint_{Q_{\rho,\theta}}|Du|^p\,\dx\dt
    +\frac{C \varepsilon}{\delta^{\frac{2-p}{p}}}  
   \left[
    \sfI_1
    +
    \iint_{Q_{\rho,\theta}}
    |Du-Dv|^p\,\dx\dt
    \right]\\
    &\phantom{\le\,}
+
\frac{C}{\delta^{\frac{2-p}{p}}\varepsilon^{\frac{N+p}{p(N+2)-N-p}}}
\left[
\iint_{Q_{\rho,\theta}}
|\sff|^{q'}
\,\dx\dt
\right]^{\frac{p(N+2)-N}{p(N+2)-N-p}}.
\end{align*}
Choosing 
\[
\varepsilon = \frac{1}{2C}\,\delta^{\frac{2-p}{p}},
\]
the second term on the right-hand side can be absorbed into the left-hand side. 
We thus arrive at
\begin{align*}
\sfI_1&
+
\iint_{Q_{\rho,\theta}}
|Du-Dv|^p\,\dx\dt\\
&\le
6\delta
\iint_{Q_{\rho,\theta}}
|Du|^p\,\dx\dt
+
\frac{C}{\delta^{
\frac{(2-p)(N+2)}{p(N+2)-N-p}}}
\left[
\iint_{Q_{\rho,\theta}}
|\sff|^{q'}
\,\dx\dt
\right]^{\frac{p(N+2)-N}{p(N+2)-N-p}} ,
\end{align*}
where $C=C(N,p,C_o)$.
\end{proof}

\section{Proof of the gradient estimate}

In this section we prove the gradient estimate. We proceed in several steps. 

\subsection{Stopping time argument.} 
We work in the following setting. 
Let $Q_R(\tilde z_o)=B_R(\tilde x_o)\times (\tilde t_o-R^2,\, \tilde t_o+R^2)
\Subset E_T$,
and, without loss of generality, assume that the center 
$\tilde z_o=(\tilde x_o, \tilde t_o)$ coincides with the origin.  We consider nested cylinders
$Q_{\frac12 R}\subset Q_{r_1}
\subset Q_{r_2}\subset Q_R$, with
$\frac12 R \le r_1 < r_2 \le R$. We define $\lambda_o>1$ by
\begin{align}\label{def:lamdda_0}
\nonumber
    \lambda_o^\frac{p}{d}
    &:=
    1+
    \biint_{Q_R}|Du|^p\,\dx\dt\\
    &\phantom{:=1\,\,}
    +
    {\sf M}^p|Q_R|^{\frac{p}{p(N+2)-(N+p)}}
    \bigg[
    \biint_{Q_R}|\sff|^{q'}\,\dx\dt
    \bigg]^{\frac{1}{q'}\frac{p(N+2)}{p(N+2)-(N+p)}},
\end{align}
where $d$ is defined in~\eqref{deficit} and
\begin{equation*}
    \qquad q=p\frac{N+2}{N},
    \qquad
    q'=\frac{p(N+2)}{p(N+2)-N}.
\end{equation*}
The parameter $\sf M>1$ will be specified later, 
depending only on $N$, $p$, $C_o$, and $C_1$. Note that the exponents of $|Q_R|$ and of the averaged quantity 
$\iint_{Q_R} |\sff|^{q'}\,\dx\dt$ are related through the identity
\begin{equation*}
    \frac{p}{p(N+2)-(N+p)}
    =
    \frac{1}{q'}\frac{p(N+2)}{p(N+2)-(N+p)} -1.
\end{equation*}
Let $\lambda  >1$ and $\sfm = \max \big\{ 1 , \lambda^\frac{2-p}{2} \big\}$. We observe that for any point $z_o\in Q_{r_1}$, and any
\begin{equation}\label{choice-r}
\frac{r_2-r_1}{2^{j+1}\sfm}
\le r \le
\frac{r_2-r_1}{2 \sfm},
\end{equation}
the intrinsic cylinder
\[
Q_r^{(\lambda)}(z_o)
:=
B_r(x_o)\times
\big(t_o-\lambda^{2-p}r^2,\,
t_o+\lambda^{2-p}r^2\big)
\]
is contained in $Q_{r_2}$, since 
$$
\lambda^{2-p} r^2 \leq \lambda^{2-p} \left( \frac{r_2-r_1}{2 \max\{1,\lambda^\frac{2-p}{2}\}} \right)^2 < r_2^2 - r_1^2.
$$
We define 
\begin{equation}\label{def:B}
    {\sf B}:=
    \bigg(
    \frac{2^{6}R}{r_2-r_1}
    \bigg)^{(N+2) \frac{d}{p}},
\end{equation}
and let
\begin{equation} \label{eq:lambda-geq-Blambda0}
    \lambda > {\sf B} \lambda_o.
\end{equation}
Then, enlarging the domain of integration from 
$Q_r^{(\lambda)}(z_o)$ to $Q_R$ and using 
\eqref{eq:lambda-geq-Blambda0}, we obtain for any $r$ as in \eqref{choice-r} that 
\begin{align}\label{eq:lambda-sub}\nonumber
    \biint_{Q_r^{(\lambda)}(z_o)}& |Du|^p\,\dx\dt
    +
    {\sf M}^p \big|Q_r^{(\lambda)}\big|^{\frac{p}{p(N+2)-(N+p)}}
    \bigg[
    \biint_{Q_r^{(\lambda)}(z_o)}|\sff|^{q'}
    \,\dx\dt
    \bigg]^{\frac{1}{q'}\frac{p(N+2)}{p(N+2)-(N+p)}}\\\nonumber
    &\le 
    \frac{|Q_R|}{|Q_r^{(\lambda)}|}\Bigg[
    \biint_{Q_R}|Du|^p\,\dx\dt\\\nonumber
    &\qquad\qquad\quad
    +
    {\sf M}^p|Q_R|^{\frac{p}{p(N+2)-(N+p)}}
    \bigg[
    \biint_{Q_R}|\sff|^{q'}\,\dx\dt
    \bigg]^{\frac{1}{q'}\frac{p(N+2)}{p(N+2)-(N+p)}}\Bigg]\\\nonumber
    &<  
    \Big(\frac{R}{r}\Big)^{N+2}\lambda^{p-2}\lambda_o^{\frac{p}{d}} \\
    &\leq \lambda^p.
    \end{align}
    In the final step we use \eqref{deficit}, \eqref{choice-r}, \eqref{def:B}, and \eqref{eq:lambda-geq-Blambda0} to estimate
    \begin{align*}
    \Big(\frac{R}{r}\Big)^{N+2}\lambda^{p-2}\lambda_o^{\frac{d}{p}}
    &\le
    \Big(\frac{2^{6}R}{r_2-r_1}\Big)^{N+2}
    \lambda^{p-2}\lambda_o^2 \nonumber \\
    &\le {\sf B}^2\lambda^{p-2} \Big(\frac{\lambda}{{\sf B}}\Big)^2 \nonumber \\
    &=\lambda^p,
\end{align*}
if $p\ge 2$, while for $p<2$ we have 
\begin{align*}
 \Big(\frac{R}{r}\Big)^{N+2}\lambda^{p-2}\lambda_o^\frac{p}{d} &\leq \Big(\frac{2^{6}\lambda^{\frac{2-p}{2}}R}{r_2-r_1}\Big)^{N+2}
    \lambda^{p-2}  \lambda_o^\frac{p}{d} \nonumber \\
    &\leq \Big(\frac{2^{6}R}{r_2-r_1}\Big)^{N+2}
    \lambda^{p-2 - (N+2) \frac{p-2}{2}} \Big(\frac{\lambda}{{\sf B}}\Big)^\frac{p}{d} \nonumber \\&\leq {\sf B}^\frac{p}{d}
    \lambda^{p} \lambda^{-\frac{p}{d}} \Big(\frac{\lambda}{{\sf B}}\Big)^\frac{p}{d} \nonumber \\
    &= \lambda^p.
\end{align*}
Now, let $\lambda$ be as in \eqref{eq:lambda-geq-Blambda0} and consider the {\it super-level set}
$$
    \sfE(\lambda,r_1) 
    := 
    \big\{\mbox{$z_o \in Q_{r_1}$: $z_o$ is a Lebesgue point of $ |D u|$ and $|D u|(z_o) > \lambda$}\big\}.
$$
In view of Lebesgue's differentiation theorem, for every
\(z_o\in \sfE(\lambda,r_1)\) we have
\[
    \lim_{\rho\downarrow 0}
    \biint_{Q_\rho^{(\lambda)}(z_o)} |D u|^p \,\dx\dt
    =
    |D u(z_o)|^p
    >
    \lambda^p .
\]
By the absolute continuity of the integral, and since the contribution of
the inhomogeneity is nonnegative, we conclude that
\begin{align*}
    \biint_{Q_\rho^{(\lambda)}(z_o)} &|D u|^p \,\dx\dt\\
    &+
    {\sf M}^p
    \big|Q_\rho^{(\lambda)}\big|^{\frac{p}{p(N+2)-(N+p)}}
    \bigg[
        \biint_{Q_\rho^{(\lambda)}(z_o)} |\sff|^{q'} \,\dx\dt
    \bigg]^{\frac{1}{q'}\frac{p(N+2)}{p(N+2)-(N+p)}}
    >
    \lambda^p ,
\end{align*}
for all sufficiently small \(\rho>0\).
In combination with~\eqref{eq:lambda-sub}, the absolute continuity of the
integral guarantees the existence of a maximal radius
\(\rho_{z_o}\in\big(0,\tfrac{r_2-r_1}{2^{6} \sfm}\big)\) such that
\begin{align*}
    \biint_{Q_{\rho_{z_o}}^{(\lambda)}(z_o)} &|D u|^p \,\dx\dt\\
    &+ 
    {\sf M}^p\big|Q_{\rho_{z_o}}^{(\lambda)}\big|^{\frac{p}{p(N+2)-(N+p)}} \bigg[\biint_{Q_{\rho_{z_o}}^{(\lambda)}(z_o)} |\sff|^{q'} \,\dx \dt\bigg]^{\frac{1}{q'}\frac{p(N+2)}{p(N+2)-(N+p)}}  
    =
    \lambda^p, 
\end{align*}
and 
\begin{align*}
    \biint_{Q_{\rho}^{(\lambda)}(z_o)}& |D u|^p \,\dx\dt \\
    &+ 
    {\sf M}^p\big|Q_{\rho}^{(\lambda)}\big|^{\frac{p}{p(N+2)-(N+p)}} \bigg[\biint_{Q_{\rho}^{(\lambda)}(z_o)} |\sff|^{q'} \,\dx \dt\bigg]^{\frac{1}{q'}\frac{p(N+2)}{p(N+2)-(N+p)}}  
    <
    \lambda^p
\end{align*}
for every $\rho\in \big(\rho_{z_o}, \frac{r_2-r_1}{2\sfm}\big]$.

\subsection{Covering of super-level sets}
As before, let \(\lambda\) satisfy~\eqref{eq:lambda-geq-Blambda0} and
consider the family of intrinsic cylinders
\begin{equation*}
    \mathcal{F}
    :=
    \big\{
        Q_{\rho_{z_o}}^{(\lambda)}(z_o)
        \colon
        z_o\in \sfE(\lambda,r_1)
    \big\},
\end{equation*}
constructed in the previous step.
By Vitali's covering theorem, we can extract a countable subfamily of
pairwise disjoint cylinders \(
\big\{Q_i\equiv Q_{\rho_{z_i}}^{(\lambda)}(z_i)\big\}_{i\in\mathbb N}\subset\mathcal F\)
such that
\[
    \sfE(\lambda,r_1)
    \subset
    \bigcup_{i\in\mathbb N} 
    Q_{5\rho_{z_i}}^{(\lambda )}(z_i).
\]
In particular, for every \(i\in\mathbb N\) we have
\(z_i\in \sfE(\lambda,r_1)\) and
\(\rho_i:=\rho_{z_i}\in\big(0,\tfrac{r_2-r_1}{2^{6} \sfm}\big)\), with
\(\rho_i\) determined by the stopping-time argument, so that
\begin{align}\label{eq:intr}
    \biint_{Q_i} |D u|^p \,\dx\dt
    + 
    {\sf M}^p|Q_i|^{\frac{p}{p(N+2)-(N+p)}} \bigg[\biint_{Q_i} |\sff|^{q'} \,\dx \dt\bigg]^{\frac{1}{q'}\frac{p(N+2)}{p(N+2)-(N+p)}}  
    &=
    \lambda^p, 
\end{align}
and
\begin{align} \label{eq:stopping-time-above-maxradius}\nonumber
    \biint_{Q_{\rho}^{(\lambda)}(z_i)}& |D u|^p \,\dx\dt \\
    &+ 
    {\sf M}^p\big|Q_{\rho}^{(\lambda)}\big|^{\frac{p}{p(N+2)-(N+p)}} \bigg[\biint_{Q_{\rho}^{(\lambda)}(z_i)} |\sff|^{q'} \,\dx \dt\bigg]^{\frac{1}{q'}\frac{p(N+2)}{p(N+2)-(N+p)}}  
    <
    \lambda^p
\end{align}
for every $\rho\in \big(\rho_{i}, \tfrac{ r_2-r_1}{2 \sfm}\big]$.
Finally, we introduce the notation
\[
    2^\ell Q_i:= Q_{2^\ell\rho_i}^{(\lambda)}(z_i),
    \qquad i,\ell\in\mathbb N,
\]
and note that, by construction of 
\(\rho_i\), we have
\(2^5Q_i \subset Q_{r_2} \subset Q_R\) for every \(i\in\mathbb N\).

\subsection{Comparison with solutions to the homogeneous equation}
Let
\begin{align*}
    v_i&\in C^0
    \big(
    [t_i-\lambda^{2-p}(2^5\rho_i)^2,
    t_i+\lambda^{2-p}(2^5\rho_i)^2);
    L^2 \big(B_{2^5\rho_i}(x_i);\R^k
    \big)\big)\\
    &\phantom{\in\,}
    \cap
    L^p\big(
    t_i-\lambda^{2-p}(2^5\rho_i)^2,
    t_i+\lambda^{2-p}(2^5\rho_i)^2;
    W^{1,p} \big(B_{2^5\rho_i}(x_i);\R^k
    \big)\big)
\end{align*}
be the unique weak solution to the Cauchy-Dirichlet problem
\begin{equation} \label{sol:C-D_coefficients}
\left\{
\begin{array}{cl}
      \partial_t v_i -\Div
      \big(\sfa (x,t)|Dv_i|^{p-2}
      Dv_i\big)
      = 0,
      & 
      \mbox{in $ 2^5 Q_i$,} \\[7pt]
      v_i = u, & \mbox{on  $\partial_{\rm par} \big(2^5 Q_i\big)$.}
   \end{array}
\right.
\end{equation}
Using the comparison estimate from Lemma \ref{lem:fract-level-comparison} (in the case $p<2$, we take $\delta=1$) for the difference $u-v_i$, together with 
the sub-intrinsic property of the cylinder $2^5 Q_i$, cf. 
\eqref{eq:stopping-time-above-maxradius}, and exploiting also that 
${\sf M}>1$, we estimate the averaged integral over $2^5 Q_i$ as follows:
\begin{align*}
    \biint_{2^5 Q_i}&|Dv_i|^p\,\dx\dt\\
    &\le
    2^{p-1}\Bigg[
    \biint_{2^5 Q_i}|Dv_i-Du|^p\,\dx\dt
    +
    \biint_{2^5 Q_i}|Du|^p\,\dx\dt
    \Bigg]\\
    &\le
     2^{p}\Bigg[
    \biint_{2^5 Q_i}|Du|^p\,\dx\dt\\
    &\qquad\qquad+C
    \big| 2^5 Q_i\big|^\frac{p}{p(N+2)-(N+p)}
    \left[
    \biint_{2^5 Q_i}
    |\sff|^{q'}
    \,\dx\dt
\right]^{\frac1{q'}\frac{p(N+2)}{p(N+2)-N-p}} \Bigg]
    \\
    &\le C\Bigg[
    \biint_{2^5 Q_i}|Du|^p\,\dx\dt\\
    &\qquad\qquad+
    {\sf M}^p\big| 2^5 Q_i\big|^\frac{p}{p(N+2)-(N+p)}
    \left[
    \biint_{2^5 Q_i}
    |\sff|^{q'}
\,\dx\dt
\right]^{\frac1{q'}\frac{p(N+2)}{p(N+2)-N-p}} \Bigg]\\
&\le C\lambda^p.
\end{align*}
where $C=C(N,p,C_o)$. In particular, this shows that the cylinder $2^5 Q_i$ is sub-intrinsic with respect to $v_i$ and that the assumptions 
of the higher integrability lemma (Lemma~\ref{lem:HI}) are satisfied with 
$\sfK = C^\frac1p$. Consequently, there exists an exponent 
$\sigma>p$, depending only on $N,k,p,C_o$, and $C_1$, such that 
\begin{equation}\label{high-int-v_i}
\biint_{2^{4} Q_i}
|Dv_i|^\sigma\,\d x \d t
\le
C\lambda^\sigma,
\end{equation}
with a constant $C=C(N,k,p,C_o,C_1)$.
\subsection{Comparison with frozen-coefficient solutions}\label{sec:comp2}
Let
\begin{align*}
    w_i&\in C^0
    \big(
    [t_i-\lambda^{2-p}(2^{4}\rho_i)^2,
    t_i+\lambda^{2-p}(2^{4}\rho_i)^2);
    L^2 \big(B_{2^{4}\rho_i}(x_i);\R^k
    \big)\big)\\
    &\phantom{\in\,}
    \cap
    L^p\big(
    t_i-\lambda^{2-p}(2^{4}\rho_i)^2,
    t_i+\lambda^{2-p}(2^{4}\rho_i)^2;
    W^{1,p} \big(B_{2^{4}\rho_i}(x_i);\R^k
    \big)\big)
\end{align*}
be the unique weak solution to the Cauchy-Dirichlet problem
\begin{equation} \label{sol:C-D_fcoefficients}
\left\{
\begin{array}{cl}
      \partial_t w_i -\Div
      \big(\sfa_i |Dw_i|^{p-2}
      Dw_i\big)
      = 0,
      & 
      \mbox{in $ 2^{4} Q_i$,} \\[7pt]
      w_i = v_i, & \mbox{on  $\partial_{\rm par} \big(2^{4} Q_i\big)$,}
   \end{array}
\right.
\end{equation}
where
\begin{equation*}
    \sfa_i:=\biint_{2^{4} Q_i}\sfa(x,t)\,\dx\dt.
\end{equation*}
Using that both $v_i$ and $w_i$ are solutions to \eqref{sol:C-D_coefficients} and \eqref{sol:C-D_fcoefficients}, respectively, we have
\begin{align*}
    \partial_t(v_i-w_i)
    &-\Div \Big(\sfa_i\big(
    |Dv_i|^{p-2}Dv_i-
    |Dw_i|^{p-2}Dw_i
    \big)\Big)\\
    &=
    \Div \Big(\big( \sfa(x,t)-\sfa_i\big)
    |Dv_i|^{p-2}Dv_i\Big)
\end{align*}
in the weak sense on $2^{4}Q_i$
and $v_i-w_i=0$ on $\partial_{\rm par} \big(2^{4} Q_i\big)$.
Testing formally with
\[
\varphi := \zeta(t)(v_i-w_i),
\]
a procedure that can be justified rigorously by means of a time regularization via Steklov averages, 
where $\zeta$ is the usual time cut–off function, 
i.e. 
$\zeta \colon [t_i-\lambda^{2-p}(2^{4}\rho_i)^2,
    t_i+\lambda^{2-p}(2^{4}\rho_i)^2] \to [0,1]$ satisfies
\[
\zeta \equiv 1 \text{ on } [t_i-\lambda^{2-p}(2^{4}\rho_i)^2,t_i+\lambda^{2-p}(2^{4}\rho_i)^2 -\epsilon], 
\qquad
\zeta(t_i+\lambda^{2-p}(2^{4}\rho_i)^2) = 0,
\]
and affine on $[t_i+\lambda^{2-p}(2^{4}\rho_i)^2 -\epsilon,t_i+\lambda^{2-p}(2^{4}\rho_i)^2]$, where $0<\epsilon< 2\lambda^{2-p}(2^{4}\rho_i)^2 $, 
we obtain
\begin{align*}
\sfI\le
\underbrace{-\tfrac12 
\iint_{2^{4}Q_i}
|v_i-w_i|^2 \partial_t \zeta 
\, \dx\dt}_{\ge 0}
&+\sfI=\sfI\sfI,
\end{align*}
where
\begin{equation*}
    \sfI:=\iint_{2^{4}Q_i}\zeta
\sfa_i
\big(
|Dv_i|^{p-2}Dv_i - |Dw_i|^{p-2}Dw_i
\big)
\cdot\big(Dv_i-Dw_i\big)
\, \dx\dt
\end{equation*}
and
\begin{equation*}
 \sfI\sfI
 :=\iint_{2^{4}Q_i}\zeta
(\sfa(x,t)-\sfa_i)|Dv_i|^{p-2}Dv_i\cdot\big( Dv_i-Dw_i\big)\,\dx\dt.
\end{equation*}
In the terms $\sfI$ and $\sfI\sfI$ we let $\epsilon\downarrow 0$, so that $\zeta$ converges to the 
characteristic function of the underlying time interval. 
The expressions thus obtained, upon formally setting $\zeta=1$, 
shall again be denoted by $\sfI$ and $\sfI\sfI$. 
The corresponding inequality,
$\sfI \le \sfI\sfI$,
will be employed in two distinct respects. 
First, we use it to show that the cylinder 
$2^{4}Q_i$ is subintrinsic also with respect to $w_i$. 
This in turn allows us to invoke the a priori $L^\infty$-estimate from Lemma~\ref{lem:-apriori-p>2}
for the gradient.
To this end, we apply Lemma~\ref{lem:alg-1} 
with $Dw_i$ and $Dv_i$ in place of $\xi$ and $\zeta$, respectively. 
Observing in addition that $\sfa_i \ge C_o$, we obtain
\begin{align*}
    \iint_{2^{4}Q_i}
    |Dw_i|^p\,\dx\dt
    &\le
    C\underbrace{\iint_{2^{4}Q_i}
    |Dv_i|^p\,\dx\dt}_{\le C|2^{4}Q_i|\lambda^p}
    +\frac{C}{C_o}\sfI\\
    &\le
    C\big|2^{4}Q_i\big|\lambda^p
    +\frac{C}{C_o}\sfI\sfI.
\end{align*}
In order to estimate $\sfI\sfI$, we recall that the coefficients $\sfa$, 
and hence also $\sfa_i$, are bounded by $C_1$. 
Combining this with Young’s inequality and once more invoking the 
subintrinsic property of $v_i$, we infer
\begin{align*}
    \sfI\sfI
    &\le 2C_1
    \iint_{2^{4}Q_i}
    \big[ |Dv_i|^{p} +|Dv_i|^{p-1}|Dw_i|\big]\,
    \dx\dt\\
    &\le
    C \iint_{2^{4}Q_i}
    |Dv_i|^{p}\,\dx\dt
    +\frac{C_o}{2C} \iint_{2^{4}Q_i}
    |Dw_i|^{p}\,\dx\dt\\
    &\le
    C\big|2^{4}Q_i
    \big|\lambda^p
    +
    \frac{C_o}{2C} \iint_{2^{4}Q_i}
    |Dw_i|^{p}\,\dx\dt
\end{align*}
Inserting this estimate into the preceding bound for $Dw_i$ and 
dividing the resulting inequality by $\lvert 2^{4}Q_i\rvert$, 
we arrive at
\begin{align*}
    \biint_{2^{4}Q_i}
    |Dw_i|^p\,\dx\dt
    &\le
     C\lambda^p.
\end{align*}
This shows that the cylinder $2^{4}Q_i$ is subintrinsic also 
with respect to $w_i$. Hence, $w_i$ satisfies on 
$2^{4}Q_i$ all the assumptions of 
Lemma~\ref{lem:-apriori-p>2}. 
In particular, the constant 
$\sfK = C^{\frac1p}$ in the subintrinsic condition 
depends only on $N,k,p,C_o,$ and $C_1$. 
An application of Lemma~\ref{lem:-apriori-p>2} therefore yields
\begin{equation}\label{sup-est-w_i}
        \sup_{2^{3}Q_i} |Dw_i|
    \le
    \sfA \lambda,
\end{equation}
with a constant $\sfA = \sfA(N,k,p,C_o,C_1)$.

We now turn to the second consequence of the inequality 
$\sfI \le \sfI\sfI$, namely the comparison estimate. 
To this end, we estimate the first integral, namely $\sfI$, 
from below by exploiting the monotonicity of the map 
$\xi \mapsto |\xi|^{p-2}\xi$ and the bound 
$\sfa_i \ge C_o$, which follows directly from the ellipticity condition. 
In this way we obtain
\begin{equation*}
    \sfI
    \ge 
    C_o\iint_{2^{4}Q_i}
    \big(|Dv_i|^2+|Dw_i|^2\big)^\frac{p-2}2
    |Dv_i-Dw_i|^2\,\dx\dt. 
\end{equation*}
The second integral is estimated by means of Young's inequality 
with exponents $p'= \frac{p}{p-1}$ and $p$. This yields
\begin{align*}
    \sfI\sfI
    &\le
    \iint_{2^{4}Q_i}
    \big|\sfa(x,t)-\sfa_i\big||Dv_i|^{p-1}|Dv_i-Dw_i|\,\dx\dt\\
    &\le
    \delta^{-\frac1{p-1}}
    \iint_{2^{4}Q_i}
    \big|\sfa(x,t)-\sfa_i\big|^\frac{p}{p-1}|Dv_i|^{p}\,\dx\dt
    +\delta
    \iint_{2^{4}Q_i}|Dv_i-Dw_i|^p\,\dx\dt,
\end{align*}
for any $\delta\in(0,1]$. 
It remains to treat the first integral on the right-hand side. 
Let ${\sigma}>p$ denote the exponent arising from the higher integrability 
of $v_i$ from Lemma~\ref{lem:HI}, as provided by the quantitative reverse Hölder inequality. 
We apply Hölder's inequality with exponents $\frac{\sigma}{p}$ and 
$\frac{\sigma}{\sigma-p}$ and invoke the higher integrability estimate 
\eqref{high-int-v_i}. This yields
\begin{align*}
     \iint_{2^{4}Q_i}&
    \big|\sfa(x,t)-\sfa_i\big|^\frac{p}{p-1}|Dv_i|^{p}\,\dx\dt\\
    &\le
    \sfI\sfI\sfI
    \bigg[
    \biint_{2^{4}Q_i}|Dv_i|^\sigma\,\dx\dt
    \bigg]^\frac{p}{\sigma} \big|2^{4}Q_i\big|
    \le
    C   \big|2^{4}Q_i\big| \lambda^p \sfI\sfI\sfI
\end{align*}
where
\begin{equation*}
    \sfI\sfI\sfI
    :=
    \bigg[
    \biint_{2^{4}Q_i}
    \big|\sfa(x,t)-\sfa_i\big|^{\frac{p}{p-1}\frac{\sigma}{\sigma-p}}\,\dx\dt
    \bigg]^{\frac{\sigma-p}{\sigma}}.
\end{equation*}
To estimate $\sfI\sfI\sfI$, we use the upper bound $\sfa \le C_1$ together 
with the ${\rm VMO}$-assumption. This yields
\begin{align*}
    \sfI\sfI\sfI&\le
    (2C_1)^{p'-\frac{\sigma-p}{\sigma}}\bigg[
    \biint_{2^{4}Q_i}
    \big|\sfa(x,t)-\sfa_i\big|\,\dx\dt
    \bigg]^{\frac{\sigma-p}{\sigma}}
    \le
    (2C_1)^{p'-\frac{\sigma-p}{\sigma}} \sfv(R)^{\frac{\sigma-p}{\sigma}}.
\end{align*}
Inserting the estimate for $\sfI\sfI\sfI$ together with the quantitative 
reverse Hölder inequality from the higher integrability of $v_i$ 
into the bound for the difference term involving 
$\sfa-\sfa_i$, we obtain
\begin{align*}
     \iint_{2^{4}Q_i}
    \big|\sfa(x,t)-\sfa_i\big|^\frac{p}{p-1}|Dv_i|^{p}\,\dx\dt
    &\le
     C
   (2C_1)^{p'-\frac{\sigma-p}{\sigma}} \sfv(R)^{\frac{\sigma-p}{\sigma}} \big|2^{4}Q_i\big|\lambda^p\\
    &=
    C
    \sfv(R)^{\frac{\sigma-p}{\sigma}}
    \big|2^{4}Q_i\big|\lambda^p.
\end{align*}
Consequently, we arrive at
\begin{align*}
    \sfI\sfI
    &\le
    C\frac{\sfv(R)^{\frac{\sigma-p}{\sigma}}}{\delta^\frac1{p-1}}
    \big|2^{4}Q_i\big|\lambda^p
    +\delta
    \iint_{2^{4}Q_i}|Dv_i-Dw_i|^p\,\dx\dt.
\end{align*}
Combining this with the lower bound for $\sfI$ yields
\begin{align}\label{pre-comp}\nonumber
    C_o\iint_{2^{4}Q_i}&
    \big(|Dv_i|^2+|Dw_i|^2\big)^\frac{p-2}2
    |Dv_i-Dw_i|^2\,\dx\dt\\
    &\le
    C\frac{\sfv(R)^{\frac{\sigma-p}{\sigma}}}{\delta^\frac1{p-1}}
    \big|2^{4}Q_i\big|\lambda^p
    +\delta
    \iint_{2^{4}Q_i}|Dv_i-Dw_i|^p\,\dx\dt,
\end{align}
for any $\delta>0$. We now distinguish between the super- and subquadratic cases. 
If $p\ge 2$, the integrand of the second term on the right-hand side 
can be bounded from above by
\[
C(p)\,
\big(|Dv_i|^2+|Dw_i|^2\big)^{\frac{p-2}{2}}
|Dv_i-Dw_i|^2.
\]
This allows us to absorb that term into the left-hand side, provided 
$\delta\le \frac12 \frac{C_o}{C(p)}$. In this way, we obtain
\begin{align*}
\biint_{2^{4}Q_i}
\big(|Dv_i|^2+|Dw_i|^2\big)^{\frac{p-2}{2}}
|Dv_i-Dw_i|^2\,\dx\dt
\le
C\,\frac{\sfv(R)^\frac{\sigma-p}{\sigma}}{\delta^\frac1{p-1}}\lambda^p.
\end{align*}  
In the subquadratic case, we estimate the integrand $|Dv_i-Dw_i|^p$ in view of Lemma~\ref{lem:alg-1} by
 \begin{equation*}
     C|Dv_i|^p +C\big( |Dv_i|^2 +|Dw_i|^2\big)^\frac{p-2}2|Dv_i-Dw_i|^2.
 \end{equation*}
Consequently,
\begin{align*}
    \delta&
    \iint_{2^{4}Q_i}|Dv_i-Dw_i|^p\,\dx\dt\\
    &\le
    \delta C
    \underbrace{\iint_{2^{4}Q_i}|Dv_i|^p
    \,\dx\dt}_{\le C|2^{4}Q_i|\lambda^p}
    +C\delta
    \iint_{2^{4}Q_i}
    \big( |Dv_i|^2 +|Dw_i|^2\big)^\frac{p-2}2|Dv_i-Dw_i|^2\,\dx\dt\\
    &\le 
    C\delta\big|2^{4}Q_i\big|\lambda^p
    +
    C\delta
    \iint_{2^{4}Q_i}
    \big( |Dv_i|^2 +|Dw_i|^2\big)^\frac{p-2}2|Dv_i-Dw_i|^2\,\dx\dt.
\end{align*}
Choosing 
$
\delta \le \frac12\frac{C_o}{C},
$
the corresponding term can be absorbed into the left-hand side of 
\eqref{pre-comp}. In this way, we arrive at
\begin{equation} \label{eq:vi-wi-comparison}
    \biint_{2^{4}Q_i}
\big(|Dv_i|^2+|Dw_i|^2\big)^{\frac{p-2}{2}}
|Dv_i-Dw_i|^2\,\dx\dt
\le
C\,\bigg[\delta+\frac{\sfv(R)^{\frac{\sigma-p}{\sigma}}}{\delta^\frac1{p-1}}\bigg]\lambda^p.
\end{equation}

\subsection{Local superlevel estimates for $|Du|$}
Let $\sfA$ be the universal constant occurring in the $L^\infty$ gradient estimate
\eqref{sup-est-w_i} for $Dw_i$. Then
\begin{align} \label{eq:Dwi-Du-pointwise} 
|D w_i (z)| \leq |D u(z) - D w_i(z)| \quad \text{ for any } z \in 2^{3} Q_i \cap \{|Du| > 4\sfA \lambda\},
\end{align}
since otherwise
\begin{align*}
|D w_i (z)| \leq \sfA \lambda \leq \tfrac14 |Du(z)| \leq \tfrac14 \big[|Du(z) - Dw_i(z)| + |Dw_i(z)|\big] < \tfrac12 |Dw_i (z)|,
\end{align*}
which is a contradiction. Integrating the pointwise estimate for $|Dw_i|$ over the superlevel set
$2^{3}Q_i\cap\{|Du|>{4\sfA}\lambda\}$ and enlarging the integration
domains so that they match those appearing in the comparison estimates,
we obtain
\begin{align*}
    &\iint_{2^{3} Q_i \cap \{|Du| > 4\sfA \lambda\}}|Du|^p\,\dx\dt\\
    &\qquad\le 2^p
    \iint_{2^{3} Q_i \cap \{|Du| > 4\sfA \lambda\}}\big|Du-Dw_i\big|^p\,\dx\dt\\
    &\qquad\le
    2^{2p-1}\Bigg[
    \iint_{2^{3} Q_i }\big|Du-Dv_i\big|^p\,\dx\dt
    +
    \iint_{2^{3} Q_i \cap \{|Du| > 4\sfA \lambda\}}\big|Dv_i-Dw_i\big|^p\,\dx\dt\Bigg]\\
    &\qquad \leq C(p) 
    \iint_{2^{3} Q_i }\big|Du-Dv_i\big|^p\,\dx\dt \\
    &\qquad\quad +
    C(p) \iint_{2^{4} Q_i} \big( \big| Dw_i \big|^2 + \big|Dv_i\big|^2 \big)^\frac{p-2}{2}\big|Dv_i-Dw_i\big|^2\,\dx\dt.
\end{align*}
In case $p \geq 2$, we estimated $|Dv_i-Dw_i|^p$ from above by 
$$\big(\big|Dv_i\big|^2 + \big|Dw_i\big|^2\big)^\frac{p-2}{2} \big|Dv_i-Dw_i\big|^2.$$ 
In case $p < 2$, we used that
\begin{align*}
\big|Dv_i-Dw_i\big|^p \leq \big|Du -Dv_i\big|^p + 5^{2-p} \big(\big|Dw_i\big|^2 +\big|Dv_i \big|^2 \big)^\frac{p-2}{2} \big|Dw_i -Dv_i\big|^2
\end{align*}
holds pointwise in $2^{3} Q_i\cap \{|Du| > 4\sfA \lambda \}$. To show that this inequality is valid, it clearly suffices to consider the case $|Dv_i-Dw_i| > |Du -Dv_i|$, which together with~\eqref{eq:Dwi-Du-pointwise} implies $|Dw_i| \leq 2 |Dv_i - Dw_i|$, and furthermore,
\begin{align*}
\big(\big|Dw_i\big|^2 +\big|Dv_i\big|^2 \big)^\frac{p-2}{2} \big|Dw_i -Dv_i\big|^2 \geq 5^{p-2} \big|Dv_i - Dw_i\big|^p.
\end{align*}
Applying the corresponding comparison estimates from Lemma~\ref{lem:fract-level-comparison} and~\eqref{eq:vi-wi-comparison}, and the fact that the cylinder $2^{5} Q_i$ is subintrinsic, i.e. inequality~\eqref{eq:stopping-time-above-maxradius}, we obtain
\begin{align*}
    &\iint_{2^{3} Q_i \cap \{|Du| > 4\sfA \lambda\}}|Du|^p\,\dx\dt\\
    &\qquad\le
    \frac{C}{ \delta^{q' \frac{(2-p)_+}{p}}}
\left[
\iint_{2^5 Q_i}
|\sff|^{q'}
\,\dx\dt
\right]^{\frac1{q'}\frac{p(N+2)}{p(N+2)-N-p}}
+
C\bigg[\delta+\frac{\sfv(R)^{\frac{\sigma-p}{\sigma}}}{\delta^\frac1{p-1}}\bigg]\big|2^{5}Q_i\big|\lambda^p. 
\end{align*}
The $\sff$-integral is controlled using the fact that the cylinder
$2^5 Q_i$ is subintrinsic, that is, by \eqref{eq:stopping-time-above-maxradius} we have
\begin{align*}
    &\left[
\iint_{2^5 Q_i}
|\sff|^{q'}
\,\dx\dt
\right]^{\frac1{q'}\frac{p(N+2)}{p(N+2)-N-p}}\\
&\qquad=
\frac{\big|2^5 Q_i\big|}{{\sf M}^p}{\sf M}^p
\big|2^5 Q_i\big|^{\frac1{q'}\frac{p(N+2)}{p(N+2)-N-p}-1}
 \left[
\biint_{2^5 Q_i}
|\sff|^{q'}
\,\dx\dt
\right]^{\frac1{q'}\frac{p(N+2)}{p(N+2)-N-p}}\\
&\qquad=
\frac{\big|2^5 Q_i\big|}{{\sf M}^p}
\underbrace{
{\sf M}^p
\big|2^5 Q_i\big|^{\frac{p}{p(N+2)-N-p}}
 \left[
\biint_{2^5 Q_i}
|\sff|^{q'}
\,\dx\dt
\right]^{\frac1{q'}\frac{p(N+2)}{p(N+2)-N-p}}}_{<\lambda^p}\\
&\qquad \le
\frac{1}{{\sf M}^p}\big|2^5 Q_i\big|\lambda^p.
\end{align*}
Collecting the above estimates, we obtain for the integral of $|Du|^p$
over the superlevel set
\begin{align}\label{est:Du-super-level}
    \iint_{2^{3} Q_i \cap \{|Du| > 4\sfA \lambda\}}|Du|^p\,\dx\dt
    &\le
    C\bigg[
    \underbrace{
    \delta+
    \frac{\sfv(R)^{\frac{\sigma-p}{\sigma}}}{\delta^\frac1{p-1}}
    +
    \frac1{{\sf M}^p {\delta^{q' \frac{(2-p)_+}{p}}}}}_{:=\,{\sf Q}}\bigg]\big|Q_i\big|\lambda^p.
\end{align}
Here, we replaced the volume $|2^5 Q_i|$ on the right-hand side by $|Q_i|$, which only enlarges the corresponding constants.
Since the cylinder $Q_i$ is intrinsic, that is, \eqref{eq:intr} holds, we have
\begin{align}\label{Q_i:intrinsic}
    |Q_i|\lambda^p
    &=
    \iint_{Q_i} |D u|^p \,\dx\dt
    + 
    {\sf M}^p \bigg[\iint_{Q_i} |\sff|^{q'} \,\dx \dt\bigg]^{\frac{1}{q'}\frac{p(N+2)}{p(N+2)-(N+p)}}.
\end{align}
We now examine the structure of the exponents and the role of the ${\sf M}$-trick. 
The following considerations are heuristic and serve only to motivate the choice of the parameters.
The starting point is the quantity
\begin{equation*}
    {\sf M}^p
    \bigg[
    \iint_{Q_i} |\sff|^{q'}\,\dx\dt
    \bigg]^{\frac1{q'}\frac{p(N+2)}{p(N+2)-(N+p)}},
\end{equation*}
where
\begin{equation*}
    q'=\frac{p(N+2)}{p(N+2)-N}.
\end{equation*}
With parameters $\alpha>0$ and $\vartheta>1$ to be specified later,
we apply Hölder's inequality and obtain
\begin{align*}
    \bigg[
    \iint_{Q_i} &|\sff|^{q'}\,\dx\dt
    \bigg]^{\frac1{q'}\frac{p(N+2)}{p(N+2)-(N+p)}}\\
    &=
    \bigg[
    \iint_{Q_i} |\sff|^{\alpha}|\sff|^{q'-\alpha}\,\dx\dt
    \bigg]^{\frac1{q'}\frac{p(N+2)}{p(N+2)-(N+p)}}\\
    &\le
    \bigg[
    \iint_{Q_i} |\sff|^{\alpha\vartheta}\,\dx\dt
    \bigg]^{\frac1{ q' \vartheta }\frac{p(N+2)}{p(N+2)-(N+p)}}\bigg[
    \iint_{Q_i} |\sff|^{(q'-\alpha)\vartheta'}\,\dx\dt
    \bigg]^{\frac{1}{q'\vartheta'}\frac{p(N+2)}{p(N+2)-(N+p)}}.
\end{align*}
We choose $\vartheta$ so that the second factor has exponent one,
namely
\begin{equation*}
    \frac{1}{q'\vartheta'}\frac{p(N+2)}{p(N+2)-(N+p)}=1,
\end{equation*}
which corresponds to a scaling-consistent choice of exponents.
Substituting the explicit expression for $q'$,
this condition is equivalent to
\begin{equation*}
    \frac{\vartheta -1}{\vartheta}\frac{p(N+2)-N}{p(N+2)-(N+p)}=1.
\end{equation*}
Solving for $\vartheta$ yields
\begin{equation*}
    \vartheta =\frac{p(N+2)-N}{p}.
\end{equation*}
With this choice of $\vartheta$, we obtain
\begin{align*}
    \bigg[
    \iint_{Q_i} &|\sff|^{q'}\,\dx\dt
    \bigg]^{\frac1{q'}\frac{p(N+2)}{p(N+2)-(N+p)}}\\
    &\le
    \bigg[
    \iint_{Q_i} |\sff|^{\alpha\vartheta}\,\dx\dt
    \bigg]^{\frac1{q' \vartheta }\frac{p(N+2)}{p(N+2)-(N+p)}}
    \iint_{Q_i} |\sff|^{(q'-\alpha)\vartheta'}\,\dx\dt
    .
\end{align*}
The choice of $\alpha$ is dictated by the forthcoming Fubini-type argument
and the fact that one has to work with the exponent $\alpha\vartheta$ anyway.
Indeed, the Fubini argument effectively enlarges the exponent
$(q'-\alpha)\vartheta'$ to $(q'-\alpha)\vartheta'\frac{s}{p}$.
Accordingly, we impose
\[
\alpha\vartheta
=
\frac{\vartheta}{\vartheta-1}
\left(
\frac{p(N+2)}{p(N+2)-N}-\alpha
\right)\frac{s}{p}.
\]
A straightforward computation then gives
\[
\alpha\vartheta=\frac{s(N+2)}{N(p-1)+p+s},
\qquad
(q'-\alpha)\vartheta'=\frac{p(N+2)}{N(p-1)+p+s}.
\]
Consequently,
\begin{equation*}
(q'-\alpha)\vartheta'\frac{s}{p}
=
\alpha\vartheta.
\end{equation*}
With these choices, the previous Hölder-type estimate reduces to
\begin{align*}
    \bigg[
    \iint_{Q_i} &|\sff|^{q'}\,\dx\dt
    \bigg]^{\frac1{q'}\frac{p(N+2)}{p(N+2)-(N+p)}}\\
    &\le
    \bigg[
    \iint_{Q_i} |\sff|^{\frac{s(N+2)}{N(p-1)+p+s}}\,\dx\dt
    \bigg]^{\frac{p}{p(N+2)-(N+p)}}
    \iint_{Q_i} |\sff|^{\frac{p(N+2)}{N(p-1)+p+s}}\,\dx\dt.
\end{align*}
Recall that $\mu$ is defined in \eqref{def:mu}.
With this notation, the previous estimate shows that
\begin{align*}
    {\sf M}^p
    \bigg[
    \iint_{Q_i} & |\sff|^{q'}\,\dx\dt
    \bigg]^{\frac1{q'}\frac{p(N+2)}{p(N+2)-(N+p)}}\\
    &\le
    {\sf M}^p
     \bigg[
    \iint_{Q_i} |\sff|^{\mu s}\,\dx\dt
    \bigg]^{\frac{p}{p(N+2)-(N+p)}}
    \iint_{Q_i} |\sff|^{\mu p}\,\dx\dt\\
    &\le
    {\sf M}^p
     \bigg[
    \iint_{Q_R} |\sff|^{\mu s}\,\dx\dt
    \bigg]^{\frac{p}{p(N+2)-(N+p)}}
    \iint_{Q_i} |\sff|^{\mu p}\,\dx\dt\\
    &=
    {\sf M}^p\sfF^p\iint_{Q_i} |\sff|^{\mu p}\,\dx\dt,
\end{align*}
where
\begin{equation*}
    \sfF
    :=
    \bigg[\iint_{Q_R} |\sff|^{\mu s}
    \,\dx\dt
    \bigg]^\frac{1}{p(N+2)-(N+p)}.
\end{equation*}
Substituting this into \eqref{Q_i:intrinsic},
we obtain
\begin{align}\label{est:Q_i}
     |Q_i|\lambda^p
    &\le
    \iint_{Q_i} |D u|^p \,\dx\dt
    +
    {\sf M}^p\sfF^p\iint_{Q_i} |\sff|^{\mu p}\,\dx\dt. 
\end{align}
In the integrals on the right-hand side, we split the integration domains
into superlevel and sublevel sets using suitable parameters.
This yields
\begin{align*}
     \iint_{Q_i} |D u|^p \,\dx\dt
     &=
     \iint_{Q_i\cap\{|Du|>\lambda/4\}} |D u|^p \,\dx\dt
     +
     \iint_{Q_i\cap\{|Du|\leq \lambda/4\}} |D u|^p \,\dx\dt
     \\
     &\le
      \iint_{Q_i\cap\{|Du|> \lambda/4\}} |D u|^p \,\dx\dt
      +
      \frac{1}{4^p}
       |Q_i|\lambda^p.
\end{align*}
Similarly,
\begin{align*}
    \iint_{Q_i} |\sff|^{\mu p}\,\dx\dt
     &=
     \iint_{Q_i\cap\{|\sff|^\mu>\lambda/(4{\sf M}\sfF)\}} |\sff|^{\mu p} \,\dx\dt
     +
     \iint_{Q_i\cap\{|\sff|^\mu \le \lambda/(4{\sf M}\sfF)\}} |\sff|^{\mu p} \,\dx\dt
     \\
     &\le
      \iint_{Q_i\cap\{|\sff|^\mu>\lambda/(4{\sf M}\sfF)\}} |\sff|^{\mu p }\,\dx\dt
      +
      \frac{1}{(4{\sf M}\sfF)^p}
      |Q_i|\lambda^p.
\end{align*}
Substituting these estimates into \eqref{est:Q_i} and absorbing the
$\frac12|Q_i|\lambda^p$-term into the left-hand side yields
\begin{align*}
         |Q_i|\lambda^p
         &\le
         2
          \iint_{Q_i\cap\{|Du|>\lambda/4\}} |D u|^p \,\dx\dt+
          2{\sf M}^p\sfF^p
          \iint_{Q_i\cap\{|\sff|^\mu>\lambda/(4{\sf M}\sfF)\}} |\sff|^{\mu p }\,\dx\dt.
\end{align*}
From \eqref{est:Du-super-level} we obtain the following {\em local superlevel estimate} for $|Du|$:
\begin{align*}
    &\iint_{2^{3} Q_i \cap \{|Du| > 4\sfA \lambda\}}|Du|^p\,\dx\dt\\
    &\qquad\le
    C{\sf Q}
    \bigg[ 
     \iint_{Q_i\cap\{|Du|>\lambda/4\}} |D u|^p \,\dx\dt+
    {\sf M}^p\sfF^p
          \iint_{Q_i\cap\{|\sff|^\mu>\lambda/(4{\sf M}\sfF)\}} |\sff|^{\mu p }\,\dx\dt
    \bigg].
\end{align*}

\subsection{From local to global superlevel estimates} 
We now exploit the Vitali property of the cylinders $Q_i$. Each $Q_i$
has its center in the superlevel set $\sfE({4\sfA}\lambda,r_1)$, while
the enlarged cylinders $5Q_i$ cover this set. Taking the local
superlevel estimate into account, we therefore obtain
\begin{align*}
    \iint_{\sfE ({4\sfA}\lambda, r_1)}&|Du|^p\,\dx\dt
    \\
    &\le
    \sum_{i=1}^\infty 
    \iint_{2^{3} Q_i \cap \{|Du| > 4\sfA \lambda\}}|Du|^p\,\dx\dt\\
    &\le
    C{\sf Q}\sum_{i=1}^\infty
    \bigg[ 
     \iint_{Q_i\cap\{|Du|>\lambda/4\}} |D u|^p \,\dx\dt\\
     &\qquad\qquad\qquad+
    {\sf M}^p\sfF^p
          \iint_{Q_i\cap\{|\sff|^\mu>\lambda/(4{\sf M}\sfF)\}} |\sff|^{\mu p }\,\dx\dt
    \bigg]\\
    &\le
    C{\sf Q}
    \bigg[ 
     \iint_{\sfE(\lambda/4, r_2)}
     |D u|^p \,\dx\dt+
    {\sf M}^p\sfF^p
          \iint_{Q_{r_2}\cap\{|\sff|^\mu>\lambda/(4{\sf M}\sfF)\}} |\sff|^{\mu p }\,\dx\dt
    \bigg].
\end{align*}
\subsection{The final gradient estimate}
For $k\ge {\sf B}\lambda_o$ we work with truncated level sets.
More precisely, we introduce
\begin{equation*}
    \sfE_k(\lambda,r_1) 
    := 
    \Big\{ z \in Q_{r_1} : z \text{ is a Lebesgue point of } |D u| \text{ and } |D u|_k(z) > \lambda \Big\},
\end{equation*}
where $|Du|_k:=\min\{|Du|,k\}$ denotes the truncation of $|Du|$ at level $k$.
In terms of these sets, the previous estimate can be written as
 \begin{align*}
    \iint_{\sfE_k (4{\sfA}\lambda, r_1)}&|Du|^p\,\dx\dt
    \\
    &\le
    C{\sf Q}
    \bigg[ 
     \iint_{\sfE_k(\lambda/4, r_2)}
     |D u|^p \,\dx\dt+
    {\sf M}^p\sfF^p
          \iint_{Q_{r_2}\cap\{|\sff|^\mu>\lambda/(4{\sf M}\sfF)\}} |\sff|^{\mu p }\,\dx\dt
    \bigg].
\end{align*}
We next let $s>p$ and multiply the inequality above by $\lambda^{s-p-1}$ and integrate
with respect to $\lambda$ over the interval $({\sf B}\lambda_o,\infty)$.
This yields
\begin{align*}
\int_{{\sf B}\lambda_o}^\infty\lambda^{s-p-1} 
\int_{\sfE_k(4\sfA \lambda,r_1)} |D u|^p \, \dx\dt \, \d \lambda
&\le 
C\,{\sf Q}[\sfI+\sfI\sfI]
\end{align*}
where
\begin{align*}
    \sfI
    &:=
    \int_{{\sf B}\lambda_o}^\infty \lambda^{s-p-1} 
    \iint_{\sfE_k(\lambda/4,r_2)} |D u|^p \, \dx\dt\d\lambda
   ,
\end{align*}
and
\begin{align*}
     \sfI\sfI
     &:=
    {\sf M}^p\sfF^p
    \int_{{\sf B}\lambda_o}^\infty\lambda^{s-p-1}
    \int_{Q_{r_2} \cap \{|f|^\mu > \lambda/(8{\sf M}\sfF)\}} |\sff|^{\mu p} 
    \, \dx\dt\d \lambda.
\end{align*}
The remaining steps follow a standard argument, which we include here
for the sake of completeness and for the reader's convenience.
A Fubini-type argument shows that
\begin{align*}
    &\int_{{\sf B} \lambda_o}^\infty \lambda^{s-p-1} \iint_{\sfE_k(4\sfA\lambda,r_1)} |D u|^p \, \dx\dt \d \lambda\\
    &\qquad = 
    \frac{1}{s-p} \bigg[ (4\sfA)^{p-s} \iint_{Q_{r_1}} |Du|_k^{s-p} |D u|^p  \, 
    \dx\dt  - 
    ({\sf B} \lambda_o)^{s-p} \iint_{Q_{r_1}} |Du|^p \, 
    \dx\dt \bigg].
\end{align*}
Both $\sfI$ and $\sfI\sfI$ are treated by Fubini-type arguments.
For $\sfI$, we obtain
\begin{align*}
   \sfI
    &\leq 
    \frac{4^{s-p}}{s-p} \iint_{Q_{r_2}} |D u|_k^{s-p} |D u|^p  \, \dx\dt.
\end{align*}
For $\sfI\sfI$, we find
\begin{align*}
    \sfI\sfI
    &\leq 
    \frac{8^{s-p}}{s-p}{\sf M}^s\sfF^s
    \iint_{Q_{r_2}} |\sff|^{\mu s} \, \dx\dt.
\end{align*}
Combining these estimates results in
\begin{align*}
    \iint_{Q_{r_1}} &|D u|_k^{s-p} |D u|^p  \, \dx\dt \\
    &\leq 
    C_\ast \sfA^{s-p} {\sf Q} \iint_{Q_{r_2}} |D u|_k^{s-p} |Du|^p  \, \dx\dt\\
    &\phantom{\le\,} +
     (\sfA{\sf  B}\lambda_o)^{s-p} \iint_{Q_{r_1}} |D u|^p \, \dx\dt
     + 
    C_\ast \sfA^{s-p} {\sf Q} 
   \mathsf M^{s} \sfF^s \iint_{Q_{r_2}} |\sff|^{\mu s} \, \dx\dt,
\end{align*}
where
\begin{equation*}
    {\sf Q}
    =
    \delta+
    \frac{\sfv(R)^{\frac{\sigma-p}{\sigma}}}{\delta^\frac1{p-1}}
    +
    \frac1{{\sf M}^p{ \delta^{q' \frac{(2-p)_+}{p}}}}.
\end{equation*}
The strategy is now to absorb the first term on the right-hand side
into the left. This requires suitable choices of the parameters
appearing in ${\sf Q}$, namely $\delta$, $R$, and ${\sf M}$.
We begin by choosing $\delta$ such that
\begin{equation*}
    C_\ast \sfA^{s-p}\,\delta = \tfrac16,
    \qquad\text{that is,}\qquad
    \delta = \frac{1}{6C_\ast \sfA^{s-p}}.
\end{equation*}
This fixes $\delta$ in terms of $p$, $s$, $C_\ast$, and ${\sfA}$. Next, we choose ${\sf M}$ sufficiently large such that
\begin{equation*}
    C_\ast {\sfA}^{s-p}{\sf M}^{-p} { \delta^{-q' \frac{(2-p)_+}{p}} }= \tfrac16,
    \qquad\text{that is,}\qquad
    {{\sf M} = \big[6C_\ast {\sfA}^{s-p}\big]^{\frac{1}{p}(q' \frac{(2-p)_+}{p} +1)}}.
\end{equation*}
This fixes ${\sf M}$ in terms of $p$, $s$, $C_\ast$, and ${\sfA}$. Finally, we choose $R_o>0$ to be sufficiently small such that
\begin{equation*}
C_\ast {\sfA}^{s-p}
\frac{\sfv(R)^{\frac{\sigma-p}{\sigma}}}{\delta^{\frac1{p-1}}}
\le \tfrac16
\end{equation*}
for all $0<R\le R_o$. With these choices, we have
\[
C_\ast {\sfA}^{s-p}{\sf Q}\le \tfrac12,
\]
and the previous estimate becomes
\begin{align*}
    \iint_{Q_{r_1}} &|D u|_k^{s-p} |D u|^p  \, \dx\dt\\
    &\leq 
    \tfrac12 \iint_{Q_{r_2}} |D u|_k^{s-p} |Du|^p  \, \dx\dt +
    \frac{C\lambda_o^{s-p} R^\beta}{(r_2-r_1)^\beta}
      \iint_{Q_{R}} |D u|^p \, \dx\dt\\
     &\phantom{\le\,}
     + 
    \tfrac12 
   \mathsf M^{s} \sfF^s \iint_{Q_{R}} |\sff|^{\mu s} \, \dx\dt,
\end{align*}
where $\beta:=\frac{d}{p}(N+2)(s-p)$. In passing to the last line, we have also used the definition of $\sf B$
in~\eqref{def:B}. Observe that the previous estimate is valid for every
$\frac12 R \le r_1 < r_2 \le R$. Hence, we may invoke the standard
iteration lemma (see~\cite[Lemma 6.1]{Giusti}), which yields
\begin{align*}
    \iint_{Q_{\frac12 R}} &|D u|_k^{s-p} |D u|^p  \, \dx\dt\\
    &\leq 
    C\lambda_o^{s-p}
      \iint_{Q_{R}} |D u|^p \, \dx\dt
     + 
    C 
   \bigg[ \iint_{Q_{R}} |\sff|^{\mu s} \, \dx\dt\bigg]^{1+\frac{s}{p(N+2)-(N+p)}}.
\end{align*}
A straightforward computation shows that the exponent of the $|\sff|^{\mu s}$-integral can be written in terms of $\mu$ as
\begin{equation*}
1+\frac{s}{p(N+2)-(N+p)}
=
\frac1{\mu}\cdot
\frac{N+2}{N(p-1)+p}.
\end{equation*}
Using the above identity, the previous estimate can be rewritten as
\begin{align*}
    \iint_{Q_{\frac12 R}} &|D u|_k^{s-p} |D u|^p  \, \dx\dt\\
    &\leq 
    C\lambda_o^{s-p}
      \iint_{Q_{R}} |D u|^p \, \dx\dt
     + 
    C 
   \bigg[ \iint_{Q_{R}} |\sff|^{\mu s} \, \dx\dt\bigg]^{\frac1{\mu} \frac{N+2}{N(p-1)+p}}.
\end{align*}
Letting $k\to\infty$ and invoking Fatou's lemma yields
\begin{align*}
    \iint_{Q_{\frac12 R}} &|D u|^s  \, \dx\dt\leq 
    C\lambda_o^{s-p}
      \iint_{Q_{R}} |D u|^p \, \dx\dt
     + 
    C 
   \bigg[ \iint_{Q_{R}} |\sff|^{\mu s} \, \dx\dt\bigg]^{\frac1{\mu}\frac{N+2}{N(p-1)+p}}.
\end{align*}
Taking mean values on both sides and raising to the power $\frac1s$, we rewrite the inequality as
\begin{align*}
    \bigg[\biint_{Q_{\frac12 R}} &|D u|^s  \, \dx\dt\bigg]^{\frac1s} \\
    &\leq 
    C\lambda_o^{1-\frac{p}{s}}
    \bigg[ \biint_{Q_{R}} |D u|^p \, \dx\dt\bigg]^{\frac{1}{s}}
     + 
    C 
   \bigg[ \biint_{Q_{R}} |R\sff|^{\mu s} \, \dx\dt\bigg]^{\frac1{\mu s} \frac{N+2}{N(p-1)+p}} \\
    &\leq 
    C\lambda_o + 
    C 
   \bigg[ \biint_{Q_{R}} |R\sff|^{\mu s} \, \dx\dt\bigg]^{\frac1{\mu s} \frac{N+2}{N(p-1)+p}}.
\end{align*}
Here we used that $\lambda_o\ge1$ and $d\ge1$. Using the definition of $\lambda_o$ in \eqref{def:lamdda_0} and Hölder's inequality for the $\sff$-term (note that $q'\le \mu s$), we obtain
\begin{align*}
    \lambda_o
    &=
    \Bigg[1+
    \biint_{Q_R}|Du|^p\,\dx\dt
    +
    C{\sf M}^p
    \bigg[
    \biint_{Q_R}|R\sff|^{q'}\,\dx\dt
    \bigg]^{\frac{1}{q'}\frac{p(N+2)}{p(N+2)-(N+p)}} \Bigg]^\frac{d}{p} \\ 
    &\le  
    C\Bigg[
    \biint_{Q_R}|Du|^p\,\dx\dt
    +
    \bigg[
    \biint_{Q_R}|R\sff|^{\mu s}\,\dx\dt + 1
    \bigg]^{\frac{1}{\mu s}\frac{p(N+2)}{p(N+2)-(N+p)}} \Bigg]^\frac{d}{p}.
\end{align*}
Inserting this into the previous estimate and using that $d\ge1$, we obtain
\begin{align*}
    \bigg[\biint_{Q_{\frac12 R}} &|D u|^s  \, \dx\dt\bigg]^{\frac1s} \\
    &\leq 
    C\Bigg[
    \biint_{Q_R}|Du|^p\,\dx\dt
    +
    \bigg[
    \biint_{Q_R}|R\sff|^{\mu s}\,\dx\dt + 1
    \bigg]^{\frac{1}{\mu s}\frac{p(N+2)}{p(N+2)-(N+p)}} \Bigg]^\frac{d}{p}.
\end{align*}
This finishes the proof of Theorem~\ref{thm:main}. 

\section{Extensions}
As developed in \cite{Acerbi-Min}, the argument essentially relies only on the monotonicity and growth properties of the evolutionary $p$-Laplace operator, together with the self-improving property of the spatial gradient for solutions to the homogeneous system and local $L^\infty$-bounds for the gradient in the case of frozen coefficients, both under a subintrinsic condition. The former allows for the treatment of $\mathrm{VMO}$-coefficients in $(x,t)$, while the latter ensures the transfer of integrability to the gradient of the original solution. 

In particular, the arguments are robust with respect to non-divergence form right-hand sides, as they depend only on monotonicity and $p$-growth conditions together with these two properties.  Since the above framework encompasses structural conditions such as Uhlenbeck-type structure, scalar equations, and more general $p$-parabolic systems with matrix-valued coefficients, which are discussed in detail in \cite{Acerbi-Min}, we refrain from further elaboration here.

\medskip
\noindent
{\bf Acknowledgments.}  
This research was funded in whole or in part by the Austrian Science Fund (FWF) [10.55776/P36295] and [10.55776/PAT1850524]. For open access purposes, the author has applied a CC BY public copyright license to any author accepted manuscript version arising from this submission.
The second and third authors would like to acknowledge that this work is related to the “Workshop on Variational Problems and PDEs”, held on the occasion of the 60th birthday of Antonia Passarelli di Napoli, on June 19–20, 2025, at the University of Naples “Parthenope”. The fourth author was partially supported by the Early Career Grant program 2025 at the University of Salzburg.

\end{document}